\documentclass[11pt]{amsart}
\usepackage{amsfonts, amsthm, amsmath, amssymb}
\usepackage[all]{xy}
\usepackage[top=1in,bottom=1in,left=1in,right=1in]{geometry}
\usepackage{tikz}
\usepackage[pdftex]{hyperref}

\newcommand{\bC}{\mathbb{C}}

\newcommand{\bF}{\mathbb{F}}

\newcommand{\bZ}{\mathbb{Z}}
\newcommand{\unit}{\mathbf{1}}

\newcommand{\cA}{\mathcal{A}}

\newcommand{\cC}{\mathcal{C}}

\newcommand{\cI}{\mathcal{I}}

\newcommand{\cT}{\mathcal{T}}

\newcommand{\cV}{\mathcal{V}}
\newcommand{\cW}{\mathcal{W}}

\newcommand{\End}{\operatorname{End}}
\newcommand{\Ext}{\operatorname{Ext}}
\newcommand{\Extabel}{\operatorname{Extabel}}

\newcommand{\Hom}{\operatorname{Hom}}

\newcommand{\Pic}{\operatorname{Pic}}
\newcommand{\uPic}{\operatorname{\underline{Pic}}}

\newcommand{\Stab}{\operatorname{Stab}}
\newcommand{\uStab}{\operatorname{\underline{Stab}}}

\newtheorem{thm}{Theorem}[subsection]
\newtheorem{conj}[thm]{Conjecture}
\newtheorem*{thmm}{Theorem}
\newtheorem{prop}[thm]{Proposition}
\newtheorem{lem}[thm]{Lemma}
\newtheorem{cor}[thm]{Corollary}

\theoremstyle{definition}
\newtheorem{defn}[thm]{Definition}

\theoremstyle{remark}
\newtheorem{rem}[thm]{Remark}

\begin{document}

\title{Building vertex algebras from parts}
\author{Scott Carnahan}

\begin{abstract}
Given a collection of modules of a vertex algebra parametrized by an abelian group, together with one dimensional spaces of composable intertwining operators, we assign a canonical element of the cohomology of an Eilenberg-Mac Lane space.  This element describes the obstruction to locality, as the vanishing of this element is equivalent to the existence of a vertex algebra structure with multiplication given by our intertwining operators, and given existence, the structure is unique up to isomorphism.  The homological obstruction reduces to an ``evenness'' problem that naturally vanishes for 2-divisible groups, so simple currents organized into odd order abelian groups always produce vertex algebras.  Furthermore, in cases most relevant to conformal field theory (i.e., when we have well-behaved contragradients and tensor products), we obtain our spaces of intertwining operators naturally, and the evenness obstruction reduces to the question of whether the contragradient bilinear form on certain order 2 currents is symmetric or skew-symmetric.  We show that if we are given a simple regular VOA and integral-weight modules parametrized by a group of even units in the fusion ring, then the direct sum admits the structure of a simple regular VOA, called the simple current extension, and this structure is unique up to isomorphism.
\end{abstract}

\maketitle

\tableofcontents



\subsubsection{Introduction}

In this paper, we construct vertex algebras from smaller pieces, by defining a multiplication operation that connects the pieces.  The question of extending vertex algebras is about as old as the theory of vertex algebras itself, although the language of modules and intertwining operators took some time to be developed.  Lattice vertex algebras are formed from modules for the Heisenberg subalgebra \cite{B86}, and the Monster vertex operator algebra, constructed in \cite{FLM88}, was later interpreted in \cite{FFR91} and \cite{H96} as a sum of a vertex algebra and a module, with a multiplication operation defined using intertwining operators.  More recent work includes H\"ohn's reformulation of the extension problem in terms of three dimensional topological field theory \cite{H03}, and Huang-Kirillov-Lepowsky's correspondence between extensions of a regular VOA $V$ and commutative associative algebras in the vertex tensor category of $V$-modules \cite{HKL14}.

Our goal is to prove existence statements in situations where we have very little information about the fine structure of modules and intertwining operators.  We find sufficient conditions for the compatibility of intertwining operators under composition, such that the sum of modules is a vertex algebra.  From there, we seek ways to guarantee that these conditions are satisfied in as automatic a fashion as possible.  In particular, the last part of the paper is concerned with assembly from a minor generalization of simple currents, which are modules whose intertwining operators are ``as nice as possible''.

Simple current extensions were described for vertex operator algebras in \cite{DLM96}, following their introduction in the physics literature in \cite{SY89}.  Similar ideas were articulated earlier in \cite{FG88} and \cite{GW86}.  The conditions introduced in \cite{SY89} section 2 are the following: currents have simple monodromy (i.e., the intertwining operators are integral-weight and all OPEs expand in integral powers of $(z-w)$), and all currents have unique fusion rules with all other primary fields.  Translating this to vertex algebra language, we say that a $V$-module $M$ is a simple current if for any irreducible $N$, there exists a unique irreducible $X$ such that the spaces of intertwining operators are described by $I_V\binom{Y}{M,N} = \Hom_V(X,Y)$ for all $V$-modules $Y$.

There is substantial previous work in this area, with some limited existence results from twisting by weight one semi-primary elements via the operator $\Delta(z)$ (e.g., \cite{DLM96} and \cite{L97}), and much more involving explicit constructions from codes via tensor products of the Ising model (starting with \cite{M96} and \cite{DGH98}).  In these cases, the intertwining operators are explicitly known, so the existence question amounts to verifying the Jacobi identity.  More work has been done concerning the properties of such extensions: Representation theory via induced modules (under two different definitions of induced module) is studied in \cite{M98}, \cite{DL96}, \cite{L01}, and \cite{Y04}, and automorphism groups of simple current extensions are examined in \cite{S07} and \cite{GL12}.

Work on existence has been hampered by the fact that the simplest case, given by an order 2 extension, contains essentially all of the subtle complexity of the general case.  Indeed, our methods of reduction solve all problems except those that we see in the order 2 case.  More specifically, in this paper, we show that simple currents parametrized by an abelian group $A$ yield a vertex algebra as long as an evenness obstruction vanishes, and we reduce this obstruction to the triviality of a quadratic form on the elementary abelian 2-group $A/2A$.  For regular vertex operator algebras, Proposition 2.6 in \cite{LY08} reduces the existence problem in this case to the hypothesis that the canonical contragradient form on an order 2 simple current is symmetric.  While we have been unable to show that this assumption holds in general, it does hold in many cases, so we are led to the following:

\textbf{Evenness conjecture:} For any simple self-dual integral-weight module $M$ over a simple $C_2$-cofinite M\"obius vertex algebra with invariant bilinear form, the contragradient bilinear form on $M$ is symmetric.

Assuming this conjecture, we can conclude that simple current extensions always exist.  Removing the hypothesis of integral weight, we conjecture that the balancing isomorphism, given by braiding any simple current against itself, is the same as multiplication by $e^{2\pi i L(0)}$.  This general statement is also far from new, for example, it is asserted without proof for the case of modular tensor categories associated to rational CFTs in \cite{FRS04} Remark 2.15.  The analogous result in the theory of local conformal nets on $S^1$ is known as the conformal spin-statistics theorem (Theorem 3.13 of \cite{GL96}), and it was used to produce a simple current extension result (Lemma 2.1 of \cite{KL06}).  Unfortunately, a translation of Guido-Longo's argument to vertex algebra language does not seem to be straightforward.

There has been substantial progress on the evenness question, both positive and negative, after the main text of this paper was written, and it seems evenness is a much richer property than I had anticipated.  The first result in the positive direction is in \cite{vEMS}, where the evenness conjecture is shown to hold for simple rational $C_2$-cofinite self-contragradient vertex operator algebras of CFT type whose modules have group-like fusion and positive $L(0)$-spectrum.  On the negative side, Thomas Creutzig has informed me that the evenness conjecture as stated is in fact too strong, and that symplectic fermions (as an integer-graded odd simple current extension of the triplet VOA $\cW(2)$) give a counterexample even in the case of a $C_2$-cofinite vertex operator algebra.  However, there is a more general positive result in \cite{CKL15}, that the conjecture holds for vertex operator algebras that are rational, $C_2$-cofinite, and unitary.  We would be very interested to see additional results concerning satisfactory conditions for evenness and counterexamples.

The main engine behind this paper is a homological obstruction theory, using spaces of intertwining operators to construct a canonical element of $H^4(K(A,2),\bC^\times)$, for $A$ a parametrizing abelian group, and $K(A,2)$ the second Eilenberg-Mac Lane space (i.e., a 2-fold delooping).  The cohomology theory in question is far from new, dating to Mac Lane's 1950 ICM address, and its application to vertex algebra theory is also over 20 years old, starting with its role in classifying types of abelian intertwining algebras in \cite{DL93}.  We also see an application in Proposition 3.29(ii) of \cite{FRS04}, characterizing commutative Schellekens algebras in terms of the vanishing of an element of $H^4(K(A,2),\bC^\times)$.  
However, the inherent complexity of the manipulations of 4-cocycles and power series in multiple variables seems to have limited its uses, so much of our work in this paper is an attempt make such calculations avoidable in the future.

Our application of $H^4(K(A,2),\bC^\times)$ to the theory of vertex algebras differs from previous work in the following important sense: We obtain abstract existence, without requiring a collection of multiplications to be chosen ahead of time.  In contrast, \cite{DLM96} uses an explicit choice of intertwining operators, and \cite{H03} requires one to start with an abelian intertwining algebra before cutting out a subspace.  This difference is why we have to do the messy work of describing associativity and commutativity data, and why the evenness problem appears.

There are certainly plenty of directions to go beyond the content of this paper.  For example extensions by modules that aren't simple currents are a topic of active interest, especially among people seeking constructions of holomorphic VOAs of central charge 24 following Schellekens's list of candidates \cite{S93}.  The formalism of this paper allows for a minor incursion into this domain: One has a non-orbifold decomposition of the holomorphic $E_{8,1}$ VOA into a direct sum of the $G_{2,1} F_{4,1}$ VOA and an irreducible module $X$ with fusion rule $X^2 = X + 1$.  While $X$ is not a simple current, we may, by setting the intertwining operator $X \otimes X \to X((z))$ to zero, form a vertex algebra $1 + X$ with the same character as $E_{8,1}$, but where $G_{2,1} F_{4,1}$ is the algebra of fixed points of an involution.  Also, one may wish to go beyond the integral weight constraint, and construct generalized vertex algebras, where locality comes with a phase defect.  Indeed, our original motivation for this work was in generalized moonshine, where we sought to build Lie algebras by assembling twisted modules of $V^\natural$ into abelian intertwining algebras.  A special case of that construction in the order 2 case was done in \cite{H96} (modulo the sketchy proof of Theorem 3.8), and also in \cite{H12} for a different involution.  The assembly problem has been fully solved in \cite{vEMS}, and we used it in \cite{GM4} to prove Norton's Generalized Moonshine conjecture.  This paper was split off of an earlier version of that paper (which was both longer and conditional on some unavailable results) when I was informed that the integral weight case held independent interest.

\subsubsection{Main results}

Our first result gives a general sufficient criterion for existence of a vertex algebra structure.  It necessarily uses some cumbersome terminology, e.g., a commutativity datum is a collection of 1-dimensional spaces of intertwining operators that satisfy some niceness conditions, and $\cT$ is a type (e.g., weighted, quasiconformal, equipped with $G$-action).

\begin{thmm}[\ref{thm:obstruction}]
Let $V$ be a $\cT$-vertex algebra, let $A$ be an abelian group, and let $\{ M_i \}_{i \in A}$ be a set of $V$-modules in $\cT$, such that $M_0 = V$.  Given a one dimensional $\cT$-commutativity datum, and any normalized choice of nonzero elements $\{ m^{i,j}_z \in \cI_{i,j}^{i+j} \}_{i,j \in A}$, the following hold:
\begin{enumerate}
\item The action of $C^2_{ab}(A,\bC^\times)$ on $\bigoplus \cI_{i,j}^{i+j}$ induces a translation action $(\{\lambda_{i,j}\}, (F, \Omega))\mapsto (d\lambda \cdot (F, \Omega))$ on the group of abelian 3-cocycles.  This action is transitive on representatives of any fixed cohomology class in $H^3_{ab}(A,\bC^\times)$, with stabilizer given by the group of abelian 2-cocycles $\{\lambda_{i,j} \}$.
\item The function $A \to \pm 1$ defined by $i \mapsto \Omega(i,i)$ is a quadratic form invariant under $\lambda$-twist by 2-cochains.  In particular, the abelian cohomology class of $(F, \Omega)$ is canonically attached to the commutativity datum.
\item If $\{M_i\}$ is $\cT$-compatible, then the quadratic form $i \mapsto \Omega(i,i)$ is identically one if and only if there exists a normalized 2-cochain $\lambda$ such that $\{ \lambda_{i,j} m^{i,j}_z \}$ describe an $A$-graded $\cT$-vertex algebra structure on $\bigoplus_{i \in A} M_i$.
\item The $A$-graded $\cT$-vertex algebra structure on $\bigoplus_{i \in A} M_i$ is unique up to isomorphism if it exists.
\end{enumerate}
\end{thmm}

Our second result requires an evenness condition on simple currents (and in particular the elements of the fusion ring parametrizing them).

\begin{thmm}[\ref{thm:even-units-yield-extension}]
Let $V$ be a $C_2$-cofinite simple M\"obius vertex algebra with finite dimensional weight spaces and invariant bilinear form.  For any collection of $V$-modules parametrized by a group of even units in the fusion ring of $V$, the direct sum has the structure of a simple current extension, and this structure is unique up to isomorphism.  The simple current extension is a $C_2$-cofinite M\"obius vertex algebra with finite dimensional weight spaces.  If $V$ is also a rational vertex operator algebra, then the simple current extension is also rational, i.e., it is a simple regular vertex operator algebra.
\end{thmm}

Our third result is conditional on the evenness conjecture.

\begin{thmm}[\ref{thm:conditional-result}]
Let $V$ be a $C_2$-cofinite simple M\"obius vertex algebra with finite dimensional weight spaces and invariant bilinear form.  Assuming the evenness conjecture, any collection of integral-weight modules parametrized by a group of units in the fusion ring of $V$ admits the structure of a simple current extension on the direct sum.  The simple current extension is a simple $C_2$-cofinite M\"obius vertex algebra, and it is unique up to isomorphism.  If $V$ is also a rational vertex operator algebra, then the simple current extension is also rational, i.e., it is a simple regular vertex operator algebra.
\end{thmm}

\subsubsection{Acknowledgements}

The author would like to thank Tomoyuki Arakawa, Bruce Bartlett, Richard Borcherds, Thomas Creutzig, Yoshitake Hashimoto, Andr\'e Henriques, Gerald H\"ohn, Satoshi Kondo, Ching-Hung Lam, Jacob Lurie, Chris Schommer-Pries, Hiroki Shimakura, Hiromichi Yamada, and the anonymous referee for helpful conversations and advice.

This research was partly supported by NSF grant DMS-0354321,  JSPS Kakenhi Grant-in-Aid for Young Scientists (B) 24740005, the Program to Disseminate Tenure Tracking System, MEXT, Japan, and the World Premier International Research Center Initiative (WPI Initiative), MEXT, Japan.

\section{Notation and definitions}

\subsection{Formal power series}

We will describe basic aspects of power series.

Concretely, we will encounter situations similar to the following: we are given maps from a vector space $W$ to the formal power series vector spaces $V((z))((w))$ and $V((w))((z))$, and we would like to say that both maps factor through a common subspace $V[[z,w]][(z-w)^{-1}]$.  To make sense of such a claim, we need to make an unambiguous choice of inclusions from the latter vector space to the former two.  Since this paper is only concerned with power series that have integer exponents, there isn't much ambiguity in the first place.

The methods of formal calculus (see e.g., \cite{FLM88} Chapter 8) provide one way to make such a selection: To expand in the first space, we write $(z-w)^n$ as $z^n(1-w/z)^n$, and use the binomial theorem to get the power series $z^n \sum_{k\geq 0} \binom{n}{k} w^k/z^k$.  To expand in the second space, we identify $(z-w)^n$ with $(-(w-z))^n$, expand as $(-1)^n w^n(1-z/w)^n$, and apply the binomial theorem as before.

Because we only need power series in three variables in this paper, we will not formulate a completely general theory.  Instead, we will explicitly enumerate the embeddings we will need, and we will check that our conventions are consistent.

\begin{defn}
\begin{enumerate}
\item We consider power series only in the 6 fundamental coordinates $z$, $w$, $t$, $z-w$, $z-t$, and $w-t$.
\item Given a fundamental coordinate such as $z$, we use the notation $V \mapsto V[[z^{\pm 1}]]$, $f \mapsto f(z)$ to denote the endofunctor on the category of complex vector spaces that takes a vector space $V$ to the space $V[[z^{\pm 1}]]$ of integral formal power series in $z$ with coefficients in $V$, and takes a complex linear transformation $f$ to the corresponding map of power series induced by applying the transformation to each coefficient.  Elements of $V[[z^{\pm 1}]]$ can be identified with set-theoretic maps $\bZ \to V$, by identifying the coefficient of $z^n$ with the image of $n \in \bZ$ under the map.
\item Given an algebraically independent set of fundamental coordinates, we may iterate this functor to get spaces like $V[[z^{\pm 1}, (z-w)^{\pm 1}, (w-t)^{\pm 1}]]$, in this case identifiable with the space of set-theoretic maps $\bZ^3 \to V$.  All other spaces we consider will be identified as subspaces of these.
\item The formal Taylor series functor $V \mapsto V[[z]]$ takes a vector space $V$ to the space of formal Taylor series in $z$ with coefficients in $V$, which can be identified with the set of maps $\bZ \to V$ with support in $\bZ_{\geq 0}$.  Similarly, the polynomial functor $V \mapsto V[z]$ yields the subspace whose elements correspond to maps that are finitely supported, with support in $\bZ_{\geq 0}$.
\item The formal Laurent series functor $V \mapsto V((z))$ yields the space whose elements correspond to maps $\bZ \to V$ supported on $\bZ$, with support that is bounded from below.  Note that $V((z))((w))$ and $V((w))((z))$ do not form equal subspaces of $V[[z^{\pm 1},w^{\pm 1}]]$, since the powers of $w$ in the former space are bounded below, and the powers of $z$ in the latter space are bounded below, but not vice versa.
\item Partial derivative operators on series spaces act on individual fundamental coordinates.  To clarify by example, the operator $\partial_w$, as defined on $V((w))((z-w))$, takes $w^m(z-w)^n$ to $mw^{m-1}(z-w)^n$, not to $mw^{m-1}(z-w)^n - nw^m(z-w)^{n-1}$.
\end{enumerate}
\end{defn}

The reader should note that we may simplify the notation for maps when it is convenient to do so.  For example, if $L(-1): V \to V$ is a linear endomorphism of a vector space, we may denote the corresponding endomorphism on $z^nV[[z]]$ by $L(-1)$ instead of $z^nL(-1)[[z]]$.  We apologize for any confusion this may cause.

\begin{defn}
Because we are only using integral powers in this paper, we will not impose strong ordering conventions.  However, we will prefer $z-w$ to $-(w-z)$, so we will write $V((z-w))$ instead of $V((w-z))$, even though they are the same space.  We treat $w-t$ and $z-t$ similarly, i.e., one may think of this as an ordering $z \succ w \succ t$.  We obtain the following conventions.
\begin{enumerate}
\item $(z-w)^n = z^n(1-w/z)^n \in z^n \bC[z^{-1}][[w]] \subset \bC((z))((w))$, but we expand $(z-w)^n$ in $w^n \bC[w^{-1}][[z]] \subset \bC((w))((z))$ as $(-1)^n w^n(1-z/w)^n$.
\item We expand $w^n$ in $z^n \bC[z^{-1}][[z-w]] \subset \bC((z))((z-w))$ as $z^n(1-\frac{z-w}{z})n$, but we expand $z^n$ in $w^n \bC[w^{-1}][[z-w]] \subset \bC((w))((z-w))$ as $w^n(1+\frac{z-w}{w})^n$.
\item We expand $(z-t)^n$ in $(w-t)^n \bC[(w-t)^{-1}][[z-w]] \subset \bC((w-t))((z-w))$ as $(w-t)^n(1+\frac{z-w}{w-t})^n$.  Similarly, we expand $(z-w)^n$ in $(z-t)^n \bC[(z-t)^{-1}][[w-t]] \subset \bC((z-t))((w-t))$ as $(z-t)^n(1-\frac{w-t}{z-t})^n$.
\item We expand $(w-t)^n$ as $(-1)^n(z-w)^n(1-\frac{z-t}{z-w})^n$ in $(z-w)^n\bC[(z-w)^{-1}][[z-t]] \subset \bC((z-w))((z-t))$.
\end{enumerate}
\end{defn}

We define some functors from the category of vector spaces to itself using a subscript convention.  Geometrically, the subscripts describe boundary strata in a partial compactification of a moduli space of genus zero curves with marked points in the following sense: We will consider objects labeled by $i$, $j$, $k$, and $\ell$ inserted at complex points $z$, $w$, $t$, and $0$, respectively.  However, when two points collide, a complex line ``bubbles off'' at the place of their collision, and we more or less view the two insertions as lying at the same point in $\bC$ (which, by convention, we have chosen to be the point attached to the latter label), but at different points in the ``bubble''.  The corresponding vector spaces will be formal Laurent power series in the differences, and we enclose the two labels in parentheses (but we omit the parentheses enclosing the whole subscript).  For example, we have the following six formal Laurent series spaces in the fundamental variables: $V_{k\ell} = V((t))$, $V_{j\ell} = V((w))$, $V_{i\ell} = V((z))$, $V_{ik} = V((z-t))$, $V_{ij} = V((z-w))$, and $V_{jk} = V((w-t))$.  When three or more points collide, we take formal Taylor series in the differences, and invert all of the fundamental coordinates that appear - see the next definition.  We use square brackets to indicate when a set of indices may be freely permuted without changing the resulting space.

\begin{defn}
For any vector space $V$, we define $V[[z,w]][z^{-1}, w^{-1}, (z-w)^{-1}]$ as the subspace of $V((z))((w))$ in which multiplication by a sufficiently large integer power of $zw(z-w)$ yields an element of $V[[z,w]]$. This yields a subfunctor of $V \mapsto V((z))((w))$.  We write $V_{[ij]\ell} = V[[z,w]][z^{-1}, w^{-1}, (z-w)^{-1}]$, and the embedding $V[[z,w]][z^{-1}, w^{-1}, (z-w)^{-1}] \to V((z))((w))$ is written $V_{[ij]\ell} \to V_{i(j\ell)}$.  Here, the square brackets in the subscript indicate that switching $z$ and $w$ yields a naturally isomorphic functor.
\end{defn}

\begin{lem} \label{lem:basic-embeddings}
The conventions listed above yield the following embeddings:
\begin{enumerate}
\item $V_{[ij]\ell} \to V_{i(j\ell)}$, meaning $V[[z,w]][z^{-1}, w^{-1}, (z-w)^{-1}] \to V((z))((w))$
\item $V_{[ij]\ell} \to V_{j(i\ell)}$, meaning $V[[z,w]][z^{-1}, w^{-1}, (z-w)^{-1}] \to V((w))((z))$
\item $V_{[ij]\ell} \to V_{(ij)\ell}$, meaning $V[[z,w]][z^{-1}, w^{-1}, (z-w)^{-1}] \to V((w))((z-w))$
\item $V_{[ij]\ell} \to V_{(ji)\ell}$, meaning $V[[z,w]][z^{-1}, w^{-1}, (z-w)^{-1}] \to V((z))((z-w))$
\item $V_{[ij]k} \to V_{i(jk)}$, meaning $V[[z-w,w-t]][(z-w)^{-1},(z-t)^{-1},(w-t)^{-1}] \to V((z-t))((w-t))$
\item $V_{[ij]k} \to V_{(ij)k}$, meaning $V[[z-w,w-t]][(z-w)^{-1},(z-t)^{-1},(w-t)^{-1}] \to V((w-t))((z-w))$
\end{enumerate}
\end{lem}
\begin{proof}
Omitted.
\end{proof}

\begin{rem}
There are some additional expansions we need to consider when working with three variables at a time.
\begin{enumerate}
\item We expand $(z-t)^n$ in $\bC[[w,t]][w^{-1},t^{-1},(w-t)^{-1}]((z-w))$ as $((z-w)+(w-t))^n = (w-t)^n(1+\frac{z-w}{w-t})^n$.  Our choice of $(w-t)$ instead of $w$ or $t$ as the distinguished invertible fundamental coordinate is due to the fact that it yields a binomial.
\item Some trinomial expansions are equal but not obviously so.  For example $(1-\frac{w}{z} - \frac{t}{z})^n$ is equal to both $(1- \frac{t}{z})^n (1-\frac{w}{z(1-t/z)})^n$ and $(1-\frac{t(1+w/t)}{z})^n$.  These more complicated expressions may appear when composing two embeddings.
\end{enumerate}
\end{rem}

\begin{lem} \label{lem:star-of-embeddings}
Given a vector space $V$, define the following vector spaces:
\begin{enumerate}
\item $V_{[ijk]\ell} = V[[z,w,t]][z^{-1},w^{-1},t^{-1},(z-w)^{-1},(z-t)^{-1},(w-t)^{-1}]$
\item $V_{i([jk]\ell)} = V((z))[[w,t]][w^{-1},t^{-1},(w-t)^{-1}]$
\item $V_{[ij](k\ell)} = V[[z,w]][z^{-1},w^{-1},(z-w)^{-1}]((t))$
\item $V_{[i(jk)]\ell} = V[[z,t]][z^{-1},t^{-1},(z-t)^{-1}]((w-t))$
\item $V_{[(ij)k]\ell} = V[[w,t]][w^{-1},t^{-1},(w-t)^{-1}]((z-w))$
\item $V_{([ij]k)\ell} = V((t))[[z-w,z-t]][(z-w)^{-1},(z-t)^{-1},(w-t)^{-1}]$
\item $V_{i(j(k\ell))} = V((z))((w))((t))$
\item $V_{i((jk)\ell)} = V((z))((t))((w-t))$
\item $V_{(ij)(k\ell)} = V((w))((t))((z-w))$ -- a referee points out we may also use $V((w))((z-w))((t))$.
\item $V_{(i(jk))\ell} = V((t))((z-t))((w-t))$
\item $V_{((ij)k)\ell} = V((t))((w-t))((z-w))$
\end{enumerate}
Then, our conventions for expansions yield the following commutative diagram of natural transformations in vector spaces:
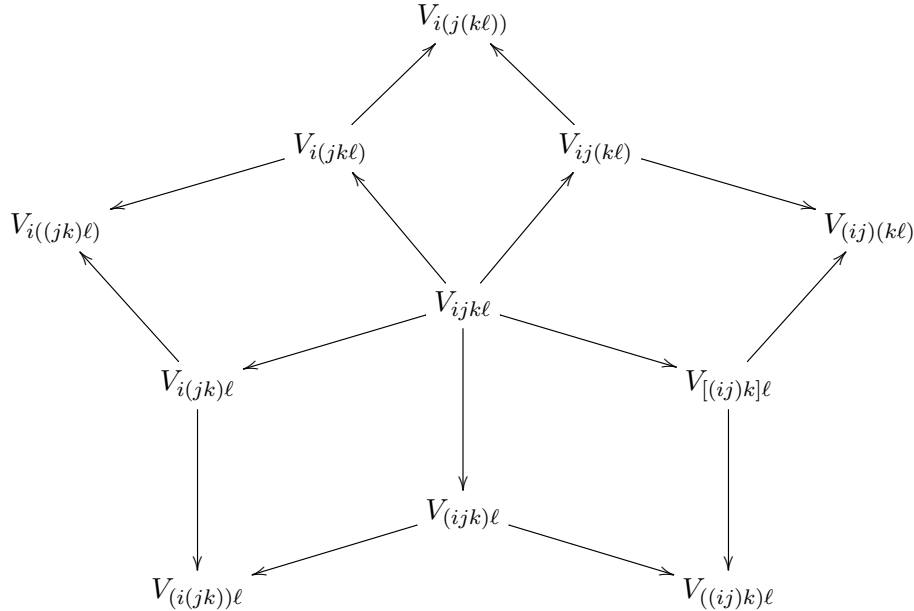
\[ \xymatrix @-1pc{
& & & V_{i(j(k\ell))} \\
& \\
& & V_{i([jk]\ell)} \ar[uur] \ar[dll] & & V_{[ij](k\ell)} \ar[uul] \ar[drr] \\
V_{i((jk)\ell)} & & & & & & V_{(ij)(k\ell)} \\
& & & V_{[ijk]\ell} \ar[ddd] \ar[drr] \ar[dll] \ar[uul] \ar[uur] \\
& V_{[i(jk)]\ell} \ar[uul] \ar[ddd] & & & & V_{[(ij)k]\ell} \ar[uur] \ar[ddd] \\
& \\
& & & V_{([ij]k)\ell} \ar[dll] \ar[drr] \\
& V_{(i(jk))\ell} & & & & V_{((ij)k)\ell} 
} \]
\end{lem}

\begin{proof}
To check the commutativity of all quadrilaterals, it suffices to expand $v z^a w^b t^c (z-w)^d (z-t)^f (w-t)^g$ along each of the compositions, where $v \in V$ and $a,b,c,d,f,g$ are integers.  Since we may check validity after pulling back all natural transformations along the unique pointed linear map $(\bC,1) \mapsto (V,v)$, we may assume without loss of generality that $V$ is $\bC$, and that $v = 1$.

The inner five arrows yield:
\begin{enumerate}
\item $z^{a+d+f} (1-w/z)^d (1-t/z)^f w^b t^c (w-t)^g \in V_{i([jk]\ell)}$
\item $z^{a+f} (1-t/z)^f w^{b+g} (1-t/w)^g (z-w)^d t^c \in V_{[ij](k\ell)}$
\item $z^a t^{b+c} (1-\frac{w-t}{t})^b (z-t)^{d+f} (1-\frac{w-t}{z-t})^d (w-t)^g \in V_{[i(jk)]\ell}$
\item $w^{a+b} (1+\frac{z-w}{w})^a t^c (w-t)^{f+g} (1+\frac{z-w}{w-t})^f (z-w)^d \in V_{[(ij)k]\ell}$
\item $t^{a+b+c} (1+\frac{z-t}{t})^a (1+\frac{w-t}{t})^b (z-w)^d (z-t)^f (w-t)^g \in V_{([ij]k)\ell}$
\end{enumerate}

We obtain five comparisons:
\begin{enumerate}
\item The top diamond yields $z^{a+d+f} w^{b+g} t^c (1-w/z)^d (1-t/z)^f (1-t/w)^g$.
\item The right diamond compares $w^{a+b+f+g} t^c (z-w)^d (1+\frac{z-w}{w})^{a+f} (1-t/w)^g (1-\frac{t}{w(1+(z-w)/w)})^f$ to $w^{a+b+f+g} (1+\frac{z-w}{w})^a t^c (1-t/w)^{f+g} (1+\frac{z-w}{w(1-t/w)})^f (z-w)^d$.  This amounts to comparing $(1+\frac{z-w}{w})^f (1-\frac{t}{w(1+(z-w)/w)})^f$ to $(1-t/w)^{f} (1+\frac{z-w}{w(1-t/w)})^{f}$, and both are expansions of $(1-\frac{t}{w} + \frac{z-w}{w})^f \in \bC[w^{-1}][[t,z-w]]$.
\item The left diamond compares $z^{a+d+f} t^{b+c} (w-t)^g (1-t/z)^{d+f} (1-\frac{w-t}{t})^b (1-\frac{w-t}{z(1-t/z)})^d$ to $z^{a+d+f} t^{b+c} (w-t)^g (1-t/z)^f (1-\frac{w-t}{t})^b (1-\frac{t(1+(w-t)/t)}{z})^d$.  This amounts to comparing $(1-t/z)^d (1-\frac{w-t}{z(1-t/z)})^d$ to $(1-\frac{t(1+(w-t)/t)}{z})^d$, and both are expansions of $(1- \frac{t}{z} - \frac{w-t}{z})^d$ in $\bC[z^{-1}][[t,w-t]]$.
\item The lower right diamond compares $t^{a+b+c} (w-t)^{f+g} (z-w)^d (1+\frac{w-t}{t})^{a+b} (1 + \frac{z-w}{t(1+(w-t)/t)})^a (1+\frac{z-w}{w-t})^f$ to $t^{a+b+c} (w-t)^{f+g} (z-w)^d (1+\frac{(z-w) + (w-t)}{t})^a (1+\frac{w-t}{t})^b (1- \frac{z-w}{w-t})^f$.  This amounts to comparing $(1 + \frac{z-w}{t(1+(w-t)/t)})^a (1+\frac{w-t}{t})^a$ to $(1+\frac{(z-w) + (w-t)}{t})^a$, and both are expansions of $(1+\frac{z-w}{t} + \frac{w-t}{t})^a$ in $\bC[t^{-1}][[z-w,w-t]]$.
\item The lower left diamond yields $t^{a+b+c} (z-t)^{d+f} (w-t)^g (1+\frac{z-t}{t})^a (1+\frac{w-t}{t})^b (1 - \frac{w-t}{z-t})^d$.
\end{enumerate}
\end{proof}

\begin{lem} \label{lem:octagons-of-embeddings}
In addition to the vector spaces from Lemma \ref{lem:star-of-embeddings}, we define the following spaces:
\begin{enumerate}
\item $V_{j((ki)\ell)} = V((w))((z))((z-t))$
\item $V_{((jk)i)\ell} = V((z))((z-t))((w-t))$
\item $V_{(j(ki))\ell} = V((z))((z-w))((z-t))$
\item $V_{j((ki)\ell)} = V((w))((z))((z-t))$
\item $V_{j((ik)\ell)} = V((w))((t))((z-t))$
\item $V_{(j(ik))\ell} = V((t))((w-t))((z-t))$
\item $V_{((ji)k)\ell} = V((t))((z-t))((z-w))$
\item $V_{(k(ij))\ell} = V((w))((w-t))((z-w))$
\item $V_{((ki)j)\ell} = V((w))((z-w))((z-t))$
\item $V_{((ik)j)\ell} = V((w))((w-t))((z-t))$
\item $V_{(i(kj))\ell} = V((w))((z-w))((w-t))$
\item $V_{i((kj)\ell)} = V((z))((w))((w-t))$
\item $V_{j([ik]\ell)} = V((w))[[z,t]][z^{-1}, t^{-1},(z-t)^{-1}]$
\item $V_{j(ki)\ell} = V[[z,w]][z^{-1}, w^{-1}, (z-w)^{-1}]((z-t))$
\item $V_{j(ik)\ell} = V[[w,t]][w^{-1}, t^{-1},(w-t)^{-1}]((z-t))$
\item $V_{([jk]i)\ell} = V((z))[[z-w,z-t]][(z-w)^{-1}, (z-t)^{-1}, (w-t)^{-1}]$
\item $V_{([ik]j)\ell} = V((w))[[z-w,z-t]][(z-w)^{-1},(z-t)^{-1},(w-t)^{-1}]$
\item $V_{i(kj)\ell} = V[[z,w]][z^{-1}, w^{-1},(z-w)^{-1}]((w-t))$
\item $V_{i([jk]\ell)} = V((z))[[w,t]][w^{-1}, t^{-1}, (w-t)^{-1}]$
\end{enumerate}
Then, we have the following commutative diagrams of embeddings:
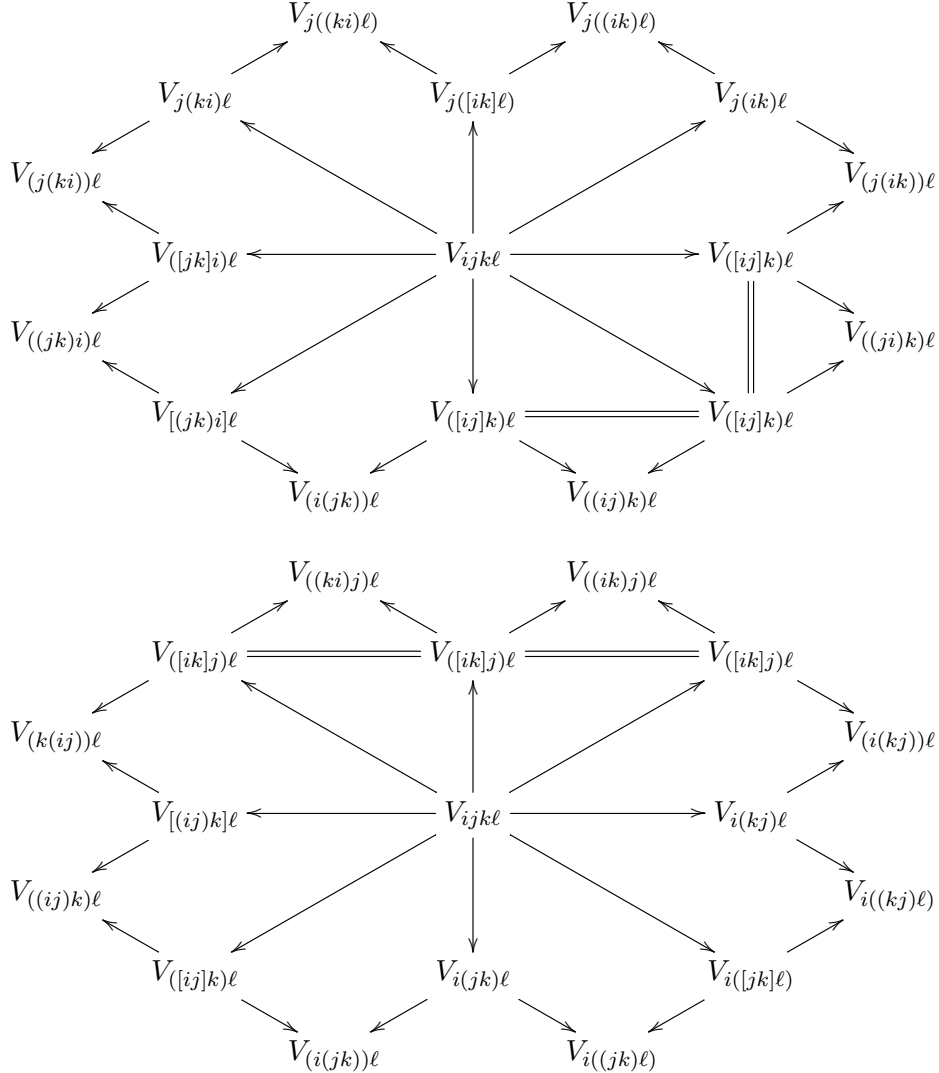
\[ \xymatrix @-1pc{
& & V_{j((ki)\ell)} & & V_{j((ik)\ell)} \\
& V_{j(ki)\ell} \ar[ur] \ar[dl] & & V_{j([ik]\ell)} \ar[ul] \ar[ur] & & V_{j(ik)\ell} \ar[ul] \ar[dr] \\
V_{(j(ki))\ell} & & & & & & V_{(j(ik))\ell} \\
& V_{([jk]i)\ell} \ar[ul] \ar[dl] & & V_{[ijk]\ell} \ar[ll] \ar[lluu] \ar[uu] \ar[uurr] \ar[rr] \ar[rrdd] \ar[dd] \ar[ddll] & & V_{([ij]k)\ell} \ar[ur] \ar[dr] \ar@{=}[dd] \\
V_{((jk)i)\ell} & & & & & & V_{((ji)k)\ell} \\
& V_{[i(jk)]\ell} \ar[ul] \ar[dr] & & V_{([ij]k)\ell} \ar[dl] \ar[dr] \ar@{=}[rr] & & V_{([ij]k)\ell} \ar[ur] \ar[dl] \\
& & V_{(i(jk))\ell} & & V_{((ij)k)\ell}
} \]

\[ \xymatrix @-1pc{
& & V_{((ki)j)\ell} & & V_{((ik)j)\ell}\\
& V_{([ik]j)\ell} \ar[ur] \ar[dl] \ar@{=}[rr] & & V_{([ik]j)\ell} \ar[ul] \ar[ur] \ar@{=}[rr] & & V_{([ik]j)\ell} \ar[ul] \ar[dr] \\
V_{(k(ij))\ell} & & & & & & V_{(i(kj))\ell} \\
& V_{[(ij)k]\ell} \ar[ul] \ar[dl] & & V_{[ijk]\ell} \ar[ll] \ar[lluu] \ar[uu] \ar[uurr] \ar[rr] \ar[rrdd] \ar[dd] \ar[ddll] & & V_{i(kj)\ell} \ar[ur] \ar[dr] \\
V_{((ij)k)\ell} & & & & & & V_{i((kj)\ell)} \\
& V_{([ij]k)\ell} \ar[ul] \ar[dr] & & V_{[i(jk)]\ell} \ar[dl] \ar[dr] & & V_{i([jk]\ell)} \ar[ur] \ar[dl] \\
& & V_{(i(jk))\ell} & & V_{i((jk)\ell)}
} \]

\end{lem}
\begin{proof}
It suffices to expand the monomial $z^a w^b t^c (z-w)^d (z-t)^f (w-t)^g$ along each of the compositions.  We will check one of the quadrilaterals, and leave the remaining 15 for the reader who wishes to repeat similar arguments.

In the first diagram, we start from $V_{[ijk]\ell}$ and go up to get $(-1)^d w^{b+d+g} z^a t^c (z-t)^f (1-z/w)^d (1-t/w)^g \in V_{j([ik]\ell)}$.  If we instead go to the upper left, we get  $(-1)^g z^{a+c} w^b (z-w)^{d+g} (1- \frac{z-t}{z})^c (1 - \frac{z-t}{z-w})^g (z-t)^f \in V_{j(ki)\ell}$.  To check commutativity, it suffices to compare $(-1)^d w^{b+d+g} z^{a+c} (z-t)^f (1-z/w)^d (1-\frac{z(1-(z-t)/z)}{w})^g (1-\frac{z-t}{z})^c$ with $(-1)^d w^{b+d+g} z^{a+c} (z-t)^f (1- \frac{z-t}{z})^c (1 - \frac{z-t}{-w(1-z/w)})^g (1-z/w)^{d+g}$ in $V_{j((ki)\ell)}$.

Both evaluate to $(-1)^d w^{b+d+g} z^{a+c} (z-t)^f (1-\frac{z-t}{z})^c (1-z/w)^d (1-\frac{z}{w} + \frac{z-t}{w})^g$, so the quadrilateral with vertex $V_{j((ki)\ell)}$ commutes.
\end{proof}

\begin{lem} \label{lem:taylor}
(Formal Taylor theorem) If $f \in V[[z^{\pm 1}]]$, then we have the following equalities in $V[[z^{\pm 1}]][[w]]$:
\begin{enumerate}
\item $f(z+w) = e^{w\partial_z}f(z) = \sum_{n \geq 0} \frac{(w)^n}{n!} \partial_z^n f(z)$.
\item $f(ze^w) = e^{wz\partial_z}f(z) = \sum_{n \geq 0} \frac{(wz)^n}{n!} \partial_z^n f(z)$.
\end{enumerate}
The analogous result holds for any algebraically independent pair of fundamental coordinates.
\end{lem}
\begin{proof}
See \cite{FLM88} Proposition 8.3.1. 
\end{proof}

\begin{rem}
One needs to be cautious about where elements live when using the formal Taylor theorem.  For example, a na\"ive expansion of $\sum_{n \geq 0} \frac{(z-w)^n}{n!} \partial_w^n (w^{-1})$ in $\bC\{z, w\}$ yields a divergent sum for the coefficient of $w^{-1}$.  However, as an element of $\bC((w))[[z-w]] \subset \bC((w))((z-w))$, we obtain the expansion of $z^{-1}$ as $\frac{1}{w} \sum_{n \geq 0} \left(\frac{z-w}{w} \right)^n$.
\end{rem}

\begin{lem} \label{lem:weird-isom} (\cite{FBZ04} Remark 5.1.4)
The map $f(z, z-w) \mapsto e^{(z-w)\partial_w} f(w,z-w) = \sum_{n \geq 0} \frac{(z-w)^n}{n!} \partial_w^n f(w,z-w)$ describes an isomorphism $V((z))((z-w)) \to V((w))((z-w))$.  Furthermore, it forms the horizontal arrow in the following commutative diagram:
\[ \xymatrix{ & V[[z,w]][z^{-1} w^{-1} (z-w)^{-1}] \ar[ld]_{\iota_{z,z-w}} \ar[rd]^{\iota_{w,z-w}} & \\
V((z))((z-w)) \ar[rr] & & V((w))((z-w)) }\]
\end{lem}
\begin{proof}
This follows straightforwardly from Lemma \ref{lem:taylor}.
\end{proof}

\subsection{Cohomology of a two-fold delooping}

In \cite{M52}, Mac Lane introduces a cohomology theory of abelian groups that is distinct from ordinary group cohomology, and describes the explicit computation of the cohomology in low degree.  Eilenberg explains in \cite{E52} that this cohomology theory is in fact the cohomology of a specific cell complex whose cohomology is isomorphic in low degree to a shift of the cohomology of the second Eilenberg-MacLane space.  This claim is proved in \cite{EM53} Theorem 20.4 (but with the definition of $A(\Pi,n)$ shifted in degree from previous papers), where they do an explicit construction of $K(A,2)$ via the diagonal of a two-fold Bar resolution.  There are other methods for explicitly describing $K(A,2)$, e.g., applying Dold-Kan correspondence to a homologically shifted abelian group (a freely available reference for this is \cite{Stacks} Theorem 019G), but the Eilenberg-Mac Lane complex happens to be the most convenient for our purposes.  

\begin{defn} (\cite{M52})
Let $A$ be an abelian group.  The Eilenberg-Mac Lane cochain complex with coefficients in $\bC^\times$ has the following description in low degree:
\[ \xymatrix{ K^1 \ar[r]^-{d^1} & K^2 = \{ f: A \times A \to \bC^\times \} \ar[r]^-{d^2} &  K^3 = \{ (F, \Omega): A^{\oplus 3} \times A^{\oplus 2} \to \bC^\times \} \ar[r]^-{d^3} & \cdots } \]
where:
\begin{enumerate}
\item $K^1 = \{ \phi: A \to \bC^\times \}$ and $K^2 = \{ f: A \times A \to \bC^\times \}$ are groups of set-theoretic maps.
\item $K^3 = \{ (F, \Omega): A^{\oplus 3} \times A^{\oplus 2} \to \bC^\times \}$ is a group of pairs of set-theoretic maps.
\item The map $d^1: K^1 \to K^2$ is given by the usual group cohomology coboundary $\phi \mapsto d^1 \phi$, defined by $d^1\phi(i,j) = \frac{\phi(j)\phi(i)}{\phi(i+j)}$.
\item The map $d^2: K^2 \to K^3$ is given by the group cohomology coboundary together with the antisymmetrizer: $d^2 f(i,j,k;\ell,m) = (\frac{f(j,k) f(i,j+k)}{f(i+j,k)f(i,j)}, \frac{f(\ell,m)}{f(m,\ell)})$.
\item The map $d^3: K^3 \to K^4$ vanishes if and only if the following conditions are satisfied: 
\begin{enumerate}
\item $F(i,j,k) F(i,j+k,\ell) F(j,k,\ell) = F(i+j,k,\ell) F(i,j,k+\ell)$ for all $i,j,k,\ell \in A$
\item $F(i,j,k)^{-1}\Omega(i,j+k) F(j,k,i)^{-1} = \Omega(i,j) F(j,i,k)^{-1} \Omega(i,k)$
\item $F(i,j,k) \Omega(i+j,k) F(k,i,j) = \Omega(j,k) F(i,k,j) \Omega(i,k)$
\end{enumerate}
\end{enumerate}
Elements annihilated by $d^n$ are called abelian $n$-cocycles, and an abelian $n$-cocycle lies in the $n$th abelian cohomology class.
\end{defn}

\begin{rem}
The reader should be aware that elements of $n$-th abelian cohomology are identified with elements of $H^{n+1}(K(A,2), \bC^\times)$.  That is, there is a shift in degrees.
\end{rem}


\begin{lem} \label{lem:cocycle-trace} (\cite{M52} Theorem 3)
If $(F, \Omega)$ is an abelian 3-cocycle, then the map $Q: A \to \bC^\times$ defined by $i \mapsto \Omega(i,i)$ is a quadratic form, i.e., $Q(i) = Q(-i)$ for all $i \in A$, and $\frac{Q(i+j+k) Q(i) Q(j) Q(k)}{Q(i+j) Q(i+k) Q(j+k)} = 1$ for all $i,j,k \in A$.  Furthermore, the trace map $(F, \Omega) \mapsto (i \mapsto \Omega(i,i))$ induces a bijection from $H^4(K(A,2), \bC^\times)$ to the set of $\bC^\times$-valued quadratic forms on $A$.
\end{lem}
\begin{proof}
This is proved as \cite{EM54}, theorem 26.1.
\end{proof}

\begin{lem} \label{lem:degree-2-abelian-cohomology}
For any abelian group $A$, we have $H^2_{ab}(A, \bC^\times) \cong H^3(K(A,2), \bC^\times) = 0$.
\end{lem}
\begin{proof}
The isomorphism between $H^3(K(A,2), \bC^\times)$ and the group $\Ext_{abel}(A,\bC^\times)$ of abelian extensions of $\bC^\times$ by $A$ is said to be ``well-known'' in  section 26 of \cite{EM54} (where the notation $\Extabel(A,\bC^\times)$ is used), and ``can be explicitly computed'' in \cite{M52} (where the confusing notation $\Ext(A, \bC^\times)$ is used).  The group is trivial because the divisible property of $\bC^\times$ makes any abelian extension split.
\end{proof}

\begin{defn} \label{defn:normalized}
Let $A$ be an abelian group.
\begin{enumerate}
\item Let $\eta: A \times A \to \bC^\times$ be a function, viewed as either a group cohomology 2-cochain or an abelian 2-cochain.  We say that $\eta$ is normalized if $\eta(0,i) = \eta(i,0) = 1$ for all $i \in A$.
\item Let $F: A \times A \times A \to \bC^\times $ be a function, viewed as a group cohomology 3-cochain.  We say that $F$ is normalized if $F(0,i,j) = F(i,0,j) = F(i,j,0) = 1$ for all $i,j \in A$.
\item We say that an abelian 3-cochain $(F, \Omega)$ is normalized if $F$ is normalized as a group cohomology 3-cochain, and $\Omega$ is normalized as a 2-cochain.
\end{enumerate}
\end{defn}

\begin{lem} \label{lem:normalized-differential}
All cohomology classes for the Eilenberg-Mac Lane model of $K(A,2)$ in degree at most 4 are represented by normalized cochains.  
Normalized cocycles in an abelian cohomology class of degree at most 4 have a transitive action by differentials of normalized cochains.
\end{lem}
\begin{proof}
Omitted - this is a straightforward computation.
\end{proof}

\begin{lem} \label{lem:pullback-quadratic-form}
Let $B \subset A$ be a subgroup, and suppose we are given a $\bC^\times$-valued quadratic form $i \mapsto \Omega(i,i)$ that is constant on cosets of $B$ in $A$.  Then this quadratic form is pulled back from a $\bC^\times$-valued quadratic form on $A/B$, and there is a representing normalized abelian 3-cocycle $(F,\Omega)$ pulled back from $A/B$.
\end{lem}
\begin{proof}
Omitted - this is also straightforward in light of \ref{lem:cocycle-trace}.
\end{proof}

\begin{lem} \label{lem:forms-with-restricted-values}
Let $Q: A \to \bC^\times$ be a quadratic form that takes values in $\pm 1$.  Then $Q$ is identically 1 on the subgroup $2A = \{ i + i | i \in A\}$, and constant on cosets of $2A$ in $A$.
\end{lem}
\begin{proof}
Both claims follow from the cube axiom: $Q(a+b+c)Q(a)Q(b)Q(c)/Q(a+b)Q(a+c)Q(b+c) = 1$.  For the first, we set $a=b=-c$, so $Q(2a) = Q(a)^4 = 1$. For the second, we set $a=b$, so $Q(2a+c) = Q(c)$.  
\end{proof}

The reader might want to know why degree 4 cohomology of $K(A,2)$ naturally arises in this paper.  Perhaps a good answer is that this cohomology group classifies braided tensor category structures on the abelian category of $A$-graded vector spaces (as shown in \cite{JS86}), and we will build something like a singular commutative ring in such a category.

\subsection{Structured vertex algebras}

\begin{defn}
Given a tensor product $A_1 \otimes \cdots \otimes A_n$ of complex vector spaces and a permutation $\sigma \in S_n$, we write $\tau_\sigma$ to indicate the linear isomorphism $A_1 \otimes \cdots \otimes A_n \to A_{\sigma(1)} \otimes \cdots \otimes A_{\sigma(n)}$ determined by the assignment $a_1 \otimes \cdots \otimes a_n \mapsto a_{\sigma(1)} \otimes \cdots \otimes a_{\sigma(n)}$, where $a_i \in A_i$ for all $1 \leq i \leq n$.  Given a set $S$ of maps $A_1 \otimes \cdots \otimes A_n \to B$, we write $\tau_\sigma^* S$ to denote the set whose elements are maps $A_{\sigma^{-1}(1)} \otimes \cdots \otimes A_{\sigma^{-1}(n)} \to B$ given by precomposing elements of $S$ with $\tau_\sigma$.
\end{defn}

\begin{defn}
A vertex algebra is a tuple $(V, \unit, L(-1), m_z)$, where:
\begin{enumerate}
\item $V$ is a vector space.
\item $\unit \in V$ is an ``identity'' or ``vacuum'' element.
\item $L(-1) \in \End(V)$ is a ``translation'' operator.
\item $m_z: V \otimes V \to V((z))$ is a linear ``multiplication'' map.
\end{enumerate}
These data are required to satisfy the following conditions:
\begin{enumerate}
\item $m_z(\unit \otimes v) = v z^0$ and $m_z(v \otimes \unit) \in v z^0 + zV[[z]]$ for all $v \in V$.
\item $L(-1) m_z(u \otimes v) - m_z(u \otimes L(-1) v) = \frac{d}{dz} m_z(u \otimes v)$ for all $u,v \in V$.  Here, the first $L(-1)$ is an abbreviation for $L(-1)((z))$.
\item The following diagram commutes:
\[ \xymatrix{ V \otimes V \otimes V \ar[rr]^{\tau_{(12)}} \ar[d]_{1 \otimes m_w} & & V \otimes V \otimes V \ar[d]^{1 \otimes m_z} \\
V \otimes V((w)) \ar[dd]_{m_z((w))} \ar[rd] & & V \otimes V((z)) \ar[dd]^{m_w((z))} \ar[dl] \\
& V[[z,w]][z^{-1},w^{-1},(z-w)^{-1}] \ar[rd] \ar[ld] \\
V((z))((w)) & & V((w))((z))
} \]
That is, the map $m_z((w)): V \otimes V((w)) \to V((z))((w))$ (resp. $m_w((z)): V \otimes V((z)) \to V((w))((z))$) factors through the natural injection $V_{[ij]\ell} \to V_{i(j\ell)}$ (resp. $V_{[ij]\ell} \to V_{j(i\ell)}$) given in Lemma \ref{lem:basic-embeddings}.
\end{enumerate}
\end{defn}

\begin{rem}
This looks slightly different from the definitions presented elsewhere (e.g., \cite{K97} section 4.9).  One typically has a map $Y: V \to (\End V)[[z^{\pm 1}]]$ in place of the map $m_z: V \otimes V \to V((z))$ - it is straightforward to show that we obtain the original axioms by replacing $m_z(u \otimes v)$ with $Y(u,z)v$.
\end{rem}

We introduce some symmetries that are commonly used in the vertex algebra literature.  These are known as ``worldsheet symmetries'' in the physics literature.

\begin{defn}
A weighted (or $\bZ$-graded) vertex algebra is a vertex algebra equipped with an operator $L(0): V \to V$ satisfying:
\begin{enumerate}
\item $L(0)$ acts semisimply on $V$ with integer eigenvalues, so that $V = \bigoplus_{n \in \bZ} V_n$ is an eigenvalue decomposition, where $L(0) v = nv$ for all $v \in V_n$.
\item $L(0) m_z(u \otimes v) - m_z(u \otimes L(0) v) = m_z(L(0) u \otimes v) + z m_z(L(-1) u \otimes v)$ for all $u,v \in V$.
\item $[L(-1), L(0)] = -L(-1)$.
\end{enumerate}
A M\"obius (resp., quasi-conformal) vertex algebra is a weighted vertex algebra equipped with an operator $L(1)$ (resp., operators $\{L(i) \}_{i = -1}^\infty$) such that:
\begin{enumerate}
\item For all $i, j \geq -1$, $[L(i),L(j)] = (i-j)L(i+j)$.  In particular, the linear span of the operators $L(-1), L(0), L(1)$ (resp., $\{L(i) \}_{i = -1}^\infty$) has a Lie algebra structure.
\item For all $i \geq -1$, $L(i) m_z(u \otimes v) - m_z(u \otimes L(i) v) = \sum_{j =0}^{i+1} \binom{i+1}{j} z^j m_z(L(i-j) u \otimes v)$.
\item The sub-Lie algebra spanned by $\{ L_i \}_{i \geq 1}$ acts locally nilpotently on $V$.
\end{enumerate}
A conformal vertex algebra of central charge $c \in \bC$ is a weighted vertex algebra equipped with a distinguished element $\omega \in V_2$ such that the operators $\{ L(n): V \to V \}_{n \in \bZ}$ defined by $m_z(\omega \otimes v) = \sum_{n \in \bZ} L(n) v z^{-n-2}$ satisfy the Virasoro relations:
\[ [L(m),L(n)] = (m-n)L(m+n) + \frac{m^3-m}{12}\delta_{m+n,0} c \]
A vertex operator algebra is a conformal vertex algebra whose $L(0)$-eigenvalues are bounded below, and whose $L(0)$-eigenspaces are finite dimensional.
\end{defn}

\begin{lem} \label{lem:vafacts}
 For future reference, we include the following facts:
\begin{enumerate}
\item For any element $v$ in a vertex algebra $V$, $m_z(v \otimes \unit) = e^{zL(-1)}v$.  More generally, we have the skew-symmetry property: $m_z(u \otimes v) = e^{zL(-1)} m_{-z}(v \otimes u) = e^{zL(-1)} m_{-z}(v \otimes u)$.  Also, we have $m_z(L(-1) u \otimes v) = \frac{d}{dz} m_z(u \otimes v)$.
\item Given a vertex algebra $V$, the following diagram commutes:
\[ \xymatrix{ & V \otimes V \otimes V \ar[ld]_{m_{z-w} \otimes 1} \ar[rd]^{1 \otimes m_w} & \\
V((z-w)) \otimes V \ar[rd] \ar[dd]_{m_w((z-w))} & & V \otimes V((w)) \ar[dl] \ar[dd]^{m_z((w))} \\
& V[[z,w]][z^{-1},w^{-1},(z-w)^{-1}] \ar[ld] \ar[rd] \\
V((w))((z-w)) & & V((z))((w)) } \]
This is called the ``associativity property''.
\item Let $V$ be a vertex algebra, let $n \geq 2$ be an integer, and let $\sigma \in S_n$ be a permutation.  Then, the map $m_{z_1} \circ (1 \otimes m_{z_2}) \circ \cdots \circ (1 \otimes \cdots \otimes 1 \otimes m_{z_n})$ from $V^{\otimes n} \otimes V$ to $V((z_n))\cdots((z_1))$ and the map
$m_{z_{\sigma(1)}} \circ (1 \otimes m_{z_{\sigma(2)}}) \circ \cdots \circ (1 \otimes \cdots \otimes 1 \otimes m_{z_{\sigma(n)}}) \circ (\tau_\sigma \otimes 1)$ from $V^{\otimes n} \otimes V$ to $V((z_{\sigma(n)}))\cdots((z_{\sigma(1)}))$ factor through equal maps from $V^{\otimes n} \otimes V$ to $V[[z_1,\ldots, z_n]][\left( \prod_{i=1}^n \left( z_i \prod_{j>i} (z_i-z_j) \right) \right)^{-1}]$.
\end{enumerate}

\end{lem}
\begin{proof}
The first claim is proved in \cite{FBZ04}, Proposition 3.2.5 and Corollary 3.1.6.

The second claim is proved in, e.g., \cite{FBZ04} Theorem 3.2.1.

The third claim is proved in \cite{FBZ04} Theorem 4.5.1.

\end{proof}

We will use the following finiteness condition near the end of the paper.
\begin{defn}
A vertex algebra is $C_2$-cofinite if the quotient $V/\{u_{-2}v|u,v \in V\}$ is a finite dimensional vector space.
\end{defn}
Weighted vertex algebras that are $C_2$-cofinite tend to be quite well-behaved.  In particular, their representation theory is highly controlled by certain finite dimensional associative algebras, although the statements of the theorems in question tend to assume vertex operator algebra structure unnecessarily.

\subsection{Modules and integral-weight intertwining operators}

\begin{defn}
Let $V$ be a vertex algebra.  A $V$-module is a vector space $M$ equipped with an action map $act_z: V \otimes M \to M((z))$ and an operator $L(-1)^M$, satisfying the following conditions:
\begin{enumerate}
\item $act_z(\unit \otimes x) = x z^0$ for all $x \in M$
\item $L(-1)^M act_z(u \otimes x) - act_z(u \otimes L(-1)^M x) = act_z(L(-1) u \otimes x)$ for all $u \in V$, $x \in M$.
\item The following diagram commutes:
\[ \xymatrix{ & V \otimes V \otimes M \ar[ld]_{m_{z-w} \otimes 1} \ar[rd]^{1 \otimes act_w} & \\
V((z-w)) \otimes M \ar[rd] \ar[dd]_{act_w((z-w))} & & V \otimes M((w)) \ar[dl] \ar[dd]^{act_z((w))} \\
& M[[z,w]][z^{-1},w^{-1},(z-w)^{-1}] \ar[ld] \ar[rd] \\
M((w))((z-w)) & & M((z))((w)) } \]
\end{enumerate}
If $V$ is weighted (resp., M\"obius, quasi-conformal), an integral-weight $V$-module is a module $M$ for the underlying vertex algebra, equipped with an operator $L(0)^M$ (resp., operators $L(0)^M$ and $L(1)^M$, operators $\{L(i)^M \}_{i \geq -1}$), satisfying the following conditions:
\begin{enumerate}
\item $L(0)^M$ acts semisimply on $M$, so that $M = \bigoplus_{n \in \bZ} M_n$ is an eigenvalue decomposition.
\item For all $i, j \geq -1$, $[L(i)^M,L(j)^M] = (i-j)L(i+j)^M$.
\item For all $i \geq -1$, any $u \in V$, and any $v \in M$, $L(i)^M act_z(u \otimes v) - act_z(u \otimes L(i)^M v) = \sum_{j =0}^{i+1} \binom{i+1}{j} z^j m_z(L(i-j) u \otimes v)$.
\item The sub-Lie algebra spanned by $\{ L_i \}_{i \geq 1}$ acts locally nilpotently on $V$.
\end{enumerate}
If $V$ is a conformal vertex algebra of central charge $c$, an integral-weight $V$-module is an integral-weight module $M$ for the underlying weighted vertex algebra, satisfying the additional condition that the operators $\{ L(n)^M: M \to M \}_{n \in \bZ}$ defined by $act_z(\omega \otimes x) = \sum_{n \in \bZ} L(n)^M x z^{-n-2}$ yield a representation of the Virasoro algebra of central charge $c$.  If $V$ is a vertex operator algebra, then an integral-weight $V$-modules is an integral-weight module $M$ for the underlying weighted vertex algebra, satisfying the additional conditions that the $L(0)$-eigenvalues are bounded below, and the $L(0)$-eigenspaces are finite dimensional.
\end{defn}

We will use the following near the end of the paper.

\begin{defn}
A vertex operator algebra $V$ is rational if every $V$-module (not just those of integral weight) is completely reducible.  $V$ is regular if it is rational and $C_2$-cofinite.
\end{defn}

\begin{defn}
Let $V$ be a vertex algebra, and let $M_1, M_2, M_3$ be $V$-modules.  An integral-weight intertwining operator of type $\binom{M_3}{M_1 M_2}$ is a map $I_z: M_1 \otimes M_2 \to M_3 ((z))$ satisfying the following conditions:
\begin{enumerate}
\item Translation covariance: $\frac{d}{dz} I_z(u \otimes v) = I_z(L(-1)^{M_1}u \otimes v) = L(-1)^{M_3} I_z(u \otimes v) - I_z(u \otimes L(-1)^{M_2} v)$ for any $u \in M_1$ and $v \in M_2$.
\item $V$-compatibility: The maps
\begin{enumerate}
\item $act_z((w)) \circ (1 \otimes I_w): V \otimes M_1 \otimes M_2 \to M_3((z))((w))$
\item $I_w((z-w)) \circ (act_{z-w} \otimes 1): V \otimes M_1 \otimes M_2 \to M_3((w))((z-w))$
\item $I_w((z)) \circ (1 \otimes act_z): M_1 \otimes V \otimes M_2 \to M_3((w))((z))$
\end{enumerate}
factor through the space $M_3[[z, w]][z^{-1}, w^{-1}, (z-w)^{-1}]$ where they coincide.
\end{enumerate}
If $V$ is weighted (resp., M\"obius, quasi-conformal), and $M_1, M_2, M_3$ are integral-weight $V$-modules in the corresponding sense, then an integral-weight intertwining operator of type $\binom{M_3}{M_1 M_2}$ is an integral-weight intertwining operator for the underlying vertex algebra modules, satisfying the compatibility:
\[ L(i)^{M_3} I_z(u \otimes v) - I_z(u \otimes L(i)^{M_2} v) = \sum_{j =0}^{i+1} \binom{i+1}{j} z^j I_z(L(i-j)^{M_1} u \otimes v) \]
for all applicable operators $L(i)$.
\end{defn}

\begin{rem}
The $V$-compatibility condition has equivalent variants in the literature, e.g., it is a straightforward translation of the commutativity and associativity given in \cite{DL93} Propositions 11.4 and 11.5, and by Remark 11.6 this is equivalent to the Jacobi identity:
\[ \begin{aligned}
 z_0^{-1} \delta \left( \frac{z_1-z_2}{z_0} \right) & Y^{M_3}(a,z_1)I(b,z_2)c - z_0^{-1} \delta \left( \frac{z_2-z_1}{z_0} \right) I(b,z_2)Y^{M_1}(a,z_1)c \\
&= z_2^{-1} \delta \left(\frac{z_1-z_0}{z_2} \right) I(Y(a,z_0)b,z_2)c. 
\end{aligned} \]
\end{rem}

\begin{lem} \label{lem:skew-sym}
Let $V$ be a vertex algebra.  Given an integral-weight intertwining operator $I_z: M_1 \otimes M_2 \to M_3 ((z))$, the map $I^*_z: M_2 \otimes M_1 \to M_3((z))$ given by $I^*_z(v \otimes u) = e^{zL(-1)} I_{-z}(u \otimes v)$ is an integral-weight intertwining operator.  In particular, for any $V$-module $M$, the map $act^*_z: M \otimes V \to M((z))$ defined by $act^*_z(x \otimes v) = e^{zL(-1)}act_{-z}(v \otimes x)$ for all $v \in V$ and $x \in M$, is an integral-weight intertwining operator.  The same statements hold if $V$ is weighted, M\"obius, or (quasi-)conformal.
\end{lem}
\begin{proof}
In the conformal case, this is Proposition 5.4.7 in \cite{FHL93}.  In the M\"obius case, this is a special case of Proposition 3.4.4 in \cite{HLZ07}, where $I_z$ is non-logarithmic, and $r=0$.  The argument in their proof does not use the M\"obius structure that is assumed to exist, except where it pertains to $L(i)$-compatibility conditions of $I^*_z$.
\end{proof}

\begin{rem}
Lemma \ref{lem:skew-sym} implies the intertwining operator $act^*_z$, together with $act_z$, induces a ``split square-zero extension'' vertex algebra structure on $V \oplus M$.
\end{rem}

\begin{lem} \label{lem:intertwining-symmetrizer}
Let $V$ be a vertex algebra, and let $M_1$ and $M_2$ be $V$-modules.  If $I_z: M_1 \otimes M_1 \to M_2\{z\}$ is a nonzero integral-weight intertwining operator that is a constant multiple of $I^*_z$, then $I_z(u \otimes v) = \pm I^*_z(u \otimes v)$ for all $u, v \in M_1$.
\end{lem}
\begin{proof}
We consider two cases, depending on whether $I_z$ is alternating.

First case: Suppose there exists $u \in M_1$ such that $I_z(u \otimes u) \neq 0$.  We may write $I_z(u \otimes u) = \sum_{n=0}^\infty g_{n+r} z^{n+r}$, with $g_r$ a nonzero element of $M_2$, and $r \in \bZ$.  Since $I^*_z(u \otimes u) = e^{zL(-1)} I_{-z}(u \otimes u)$, it suffices to check that the lowest-order term in the latter series is $(-1)^r g_r z^r$.  By assumption, $I_z$ is proportional to $I^*_z$, and our calculation shows that the constant of proportionality is $\pm 1$.

Second case: If for all $u \in M_1$, $I_z(u \otimes u) = 0$, then bilinearity implies $I_z(u \otimes v) = -I_z(v \otimes u)$ for all $u,v \in M_1$.  Write $I_z(u \otimes v) = \sum_{n=0}^\infty g_{n+r} z^{n+r}$, with $g_r$ a nonzero element of $M_2$, and $r \in \bZ$.  We note that:
\[ \begin{aligned}
I^*_z(u \otimes v) &= e^{zL(-1)} I_{-z}(v \otimes u) \\
&= - e^{zL(-1)} \sum_{n=0}^\infty g_{n+r} (-z)^{n+r}
\end{aligned} \]
so the lowest order term is $-(-1)^r g_r z^r$.  By assumption, $I_z$ is proportional to $I^*_z$, and our calculation shows that the constant of proportionality is $\pm 1$.
\end{proof}

\begin{defn}
Let $V$ be a vertex algebra equipped with an action of a group $G$ by vertex algebra automorphisms.  A $G$-action on a $V$-module $M$ (as a vector space) is compatible with the $V$-module structure if the following conditions hold:
\begin{enumerate}
\item For any element $g \in G$, $act_z(g v \otimes g u) = g act_z(v \otimes u)$ for all $v \in V$ and $u \in M$.
\item For any $g \in G$, $L(-1)g = g L(-1)$.
\item If $V$ and $M$ have additional worldsheet symmetry given by operators $\{ L(i) \}$, then $L(i) g = g L(i)$ for all applicable $i$ and all $g \in G$.
\end{enumerate}
Given $V$-modules $M_1,M_2,M_3$ equipped with $G$-actions compatible with their $V$-module structures, an integral-weight intertwining operator $I_z: M_1 \otimes M_2 \to M_3 ((z))$ is $G$-equivariant if for any $g \in G$, $I_z(g u \otimes g v) = g I_z(u \otimes v)$ for all $u \in M_1$ and $v \in M_2$.
\end{defn}

\section{Obstruction theory}

Let $A$ be an abelian group, $V$ a vertex algebra, and $\{ M_i\}_{i \in A}$ a set of $V$-modules, such that $V = M_0$.  Suppose we are given one-dimensional vector spaces $\cI_{i,j}^{i+j}$ whose elements are integral-weight intertwining operators $m^{i,j}_z: M_i \otimes M_j \to M_{i+j}((z))$ for all $i,j \in A$.  We are concerned with the following:
\vspace{5mm}

\textbf{Question:} Is it possible to choose nonzero elements in each $\cI_{i,j}^{i+j}$ so that $\bigoplus_{i \in A} M_i$ is endowed with a vertex algebra structure?
\vspace{5mm}

A priori, there is no reason for this to hold, because vertex algebras satisfy a locality property that requires compositions of our intertwining operators to satisfy some compatibilities, and we have not put in place any constraints on our maps.  In particular, the locality axiom is equivalent to the commutativity of the following diagram for each $i,j,k \in A$:

\[ \xymatrix{ M_i \otimes (M_j \otimes M_k) \ar[rr]^{\tau_{(12)}} \ar[rd]_{m^{i,j+k}_z \circ (1 \otimes m^{j,k}_w)} & & M_j \otimes (M_i \otimes M_k) \ar[dl]^{m^{j,i+k}_w \circ (1 \otimes m^{i,k}_z)} \\ & M_{i+j+k}[[z,w]][z^{-1},w^{-1},(z-w)^{-1}] } \]
and we have not even imposed the condition that the two composite maps are proportional to each other.

To convert this condition on intertwining operators into a homological problem, we rephrase it in terms of an equivalent condition that concerns singular commutativity and associativity.

Since we are interested in constructing several types of vertex algebras (e.g., M\"obius vertex algebras, quasi-conformal vertex algebras, vertex operator algebras), we will refrain from treating each case separately.  Instead, we consider a datum $\cT$, that is akin to a structured category of modules with certain spaces of integral-weight intertwining operators attached, but is used like a type.  We won't be particularly precise about the properties of $\cT$ that we need, but the results of this paper hold for any of the following examples:
\begin{enumerate}
\item The unrestricted setting: A $\cT$-vertex algebra $V$ is a vertex algebra, $V$-modules in $\cT$ are modules for the vertex algebra $V$, $\cT$-morphisms are $V$-module maps, and $\cT$-intertwining operators are integral-weight intertwining operators.
\item Worldsheet symmetry: A $\cT$-vertex algebra $V$ is a weighted (resp. M\"obius, quasiconformal, conformal) vertex algebra, $V$-modules in $\cT$ are integral-weight $V$-modules (resp. with M\"obius, quasiconformal, conformal structure), $\cT$-morphisms are $\{L(i)\}$-equivariant module maps for all applicable operators $L(i)$, and $\cT$-intertwining operators are integral-weight and compatible with the defined $L(i)$ operators.
\item VOAs: A $\cT$-vertex algebra $V$ is a vertex operator algebra, $V$-modules in $\cT$ are integral-weight $V$-modules in the VOA sense, i.e., with $L(0)$-spectrum bounded below and with finite multiplicity, morphisms are VOA-module maps, and intertwining operators are integral-weight VOA intertwining operator maps.
\item Strictly $G$-equivariant versions of the previous examples, where a group $G$ acts on $V$ by automorphisms, modules have $G$-action compatible with the action on $V$, and maps of modules and intertwining operators are $G$-equivariant.  This example would be more useful if we allowed projective actions on modules, but in our final results, we would have to replace $G$ with a central extension, and the notation would become cumbersome.
\end{enumerate}
In this language, our aim is to start with a $\cT$-vertex algebra $V$, a collection of modules $\{M_i\}_{i \in A}$ in $\cT$ with $V = M_0$, and build a $\cT$-vertex algebra structure on $\bigoplus_{i \in A} M_i$ using $\cT$-intertwining operators.  Given a $\cT$-vertex algebra $V$, and $V$-modules $M_1, M_2, M_3$ in $\cT$, we write $I_{\cT}\binom{M_3}{M_1,M_2}$ to denote the space of $\cT$-intertwining operators $M_1 \otimes M_2 \to M_3 ((z))$.  One may view the data making up $\cT$ as a rather primitive version of the notion of multicategory or pseudo-tensor category, but in the long run, some kind of type theory may be a better fit for making this example-driven study into a set of general theorems.

\subsection{Associativity}

In order for a choice of $\cT$-intertwining operators to yield a $\cT$-vertex algebra, it is necessary (but not sufficient) for the following associativity diagram to commute:
\[ \xymatrix{
& M_i \otimes M_j \otimes M_k \ar[dl]_{m^{i,j}_{z-w} \otimes 1} \ar[rd]^{1 \otimes m^{j,k}_w} \\
M_{i+j}((z-w)) \otimes M_k \ar[rd] \ar[dd]_{m^{j,k}_w} & & M_i \otimes M_{j+k}((w)) \ar[dl] \ar[dd]^{m^{i,j+k}_z} \\
& M_{i+j+k}[[z,w]][z^{-1},w^{-1},(z-w)^{-1}] \ar[ld]_{\iota_{w,z-w}} \ar[rd]_{\iota_{z,w}} \\
M_{i+j+k}((w))((z-w)) & & M_{i+j+k}((z))((w))
} \]
One way to approach the associativity condition is to assume the canonical composition operation  on $\cI^{i+j+k}_{i,j+k} \otimes \cI^{j+k}_{j,k}$ and $\cI^{i+j+k}_{i+j,k} \otimes \cI_{i,j}^{i+j}$ yields elements in the same one dimensional space $\cI^{i+j+k}_{i,j,k}$ of maps:
\[ M_i \otimes M_j \otimes M_k \to M_{i+j+k}[[z,w]][z^{-1},w^{-1},(z-w)^{-1}] \]
If this holds, then for any choice of nonzero elements in each $\cI^{i+j}_{i,j}$, the failure of the associativity diagram to commute is encoded by a set of nonzero constants $\{ F(i,j,k) \in \bC^\times \}_{i,j,k \in A}$.  We may then hope that these constants can be adjusted to 1 by suitably rescaling our intertwining operators.  To encode this more specifically, we have the following definition:

\begin{defn}
Let $V$ be a $\cT$-vertex algebra, let $A$ be an abelian group, and let $\{ M_i \}_{i \in A}$ be a set of $V$-modules in $\cT$, such that $M_0 = V$.  We define a one dimensional $\cT$-associativity datum to be a set $\{ \cI_{i,j}^{i+j} \}_{i,j \in A}$ of one-dimensional subspaces of $I_{\cT}\binom{M_{i+j}}{M_i,M_j}$ (i.e., whose elements are integral-weight intertwining operators $M_i \otimes M_j \to M_{i+j}((z))$), such that:
\begin{enumerate}
\item For each $i,j,k \in A$, there exists a one dimensional vector space $\cI_{i,j,k}^{i+j+k}$ whose elements are maps $M_i \otimes M_j \otimes M_k \to M_{i+j+k}[[z,w]][z^{-1}, w^{-1}, (z-w)^{-1}]$, such that composition of intertwining operators, combined with the canonical injections from Lemma \ref{lem:basic-embeddings}, induces isomorphisms from $\cI^{i+j+k}_{i,j+k} \otimes \cI^{j+k}_{j,k}$ to $\cI_{i,j,k}^{i+j+k}$ and from $\cI^{i+j+k}_{i+j,k} \otimes \cI_{i,j}^{i+j}$ to $\cI_{i,j,k}^{i+j+k}$.  
\item For any $i,j,k,\ell \in A$, there exists a one-dimensional space $\cI_{i,j,k,\ell}^{i+j+k+\ell}$ whose elements are maps:
\[ M_i \otimes M_j \otimes M_k \otimes M_\ell \to M_{i+j+k+\ell}[[z,w,t]][z^{-1}, w^{-1} t^{-1}, (z-w)^{-1}, (z-t)^{-1}, (w-t)^{-1}] \]
such that composition of intertwining operators, combined with one of the canonical injections from Lemma \ref{lem:star-of-embeddings}, induces an isomorphism from $\cI^{i+j+k+\ell}_{i,j+k+\ell} \otimes \cI^{j+k+\ell}_{j,k+\ell} \otimes \cI^{k + \ell}_{k,\ell}$ to $\cI_{i,j,k,\ell}^{i+j+k+\ell}$.
\item The space $\cI_{0,i}^i$ is spanned by the module structure map $act^i_z: V \otimes M_i \to M_i((z))$.
\item  The space $\cI_{i,0}^i$ is spanned by the intertwining operator $act^{i,*}_z: M_i \otimes V \to M_i((z))$ described in Lemma \ref{lem:skew-sym}.
\end{enumerate}

Given a one-dimensional $\cT$-associativity datum, and any choice of nonzero intertwining operators $\{ m^{i,j}_z \in \cI_{i,j}^{i+j} \}_{i,j \in A}$, we define the function $F: A^{\oplus 3} \to \bC^\times$ by
\[ m^{i,j+k}_z \circ (1 \otimes m^{j,k}_w) = F(i,j,k) m^{i+j,k}_w \circ (m^{i,j}_{z-w} \otimes 1) \]
\end{defn}

\begin{rem}
In order for this definition to make sense, it is necessary to choose certain embeddings of vector spaces (or more precisely, natural transformations of power series endofunctors on $Vect$), and Lemma \ref{lem:star-of-embeddings} ensures this can be done consistently.
\end{rem}

\begin{rem}
It does not seem to be necessary to assume the existence of the spaces $\cI_{i,j,k,\ell}^{i+j+k+\ell}$, as long as we assume that all of the modules $\{ M_i \}_{i \in A}$ are well-behaved simple currents (see Section \ref{sect:simple-currents}).  However, without strong conditions, it is possible for the composition of nonzero intertwining operators to be zero.
\end{rem}

\begin{lem} \label{lem:isoms}
Given a one-dimensional $\cT$-associativity datum $\{ \cI_{i,j}^{i+j} \}_{i,j \in A}$, the composition of intertwining operators induces isomorphisms between the following one-dimensional vector spaces, for all $i,j,k,\ell \in A$:
\[ \begin{array}{rcccl}
\cI_{i,j,k,\ell}^{i+j+k+\ell}
& \cong & \cI_{i,j+k+\ell}^{i+j+k+\ell} \otimes \cI_{j,k+\ell}^{j+k+\ell} \otimes \cI_{k,\ell}^{k+\ell}
& \cong &\cI_{i,j,k+\ell}^{i+j+k+\ell} \otimes \cI_{k,\ell}^{k+\ell} \\
& \cong & \cI_{i+j,k+\ell}^{i+j+k+\ell} \otimes \cI_{k,\ell}^{k+\ell} \otimes \cI_{i,j}^{i+j} 
& \cong &\cI_{i+j,k,\ell}^{i+j+k+\ell} \otimes \cI_{i,j}^{i+j} \\
& \cong & \cI_{i+j+k,\ell}^{i+j+k+\ell} \otimes \cI_{i+j,k}^{i+j+k} \otimes \cI_{i,j}^{i+j}
& \cong &\cI_{i+j+k,\ell}^{i+j+k+\ell} \otimes \cI_{i,j,k}^{i+j+k} \\
& \cong & \cI_{i+j+k,\ell}^{i+j+k+\ell} \otimes \cI_{i,j+k}^{i+j+k} \otimes \cI_{j,k}^{j+k}
& \cong &\cI_{i,j+k,\ell}^{i+j+k+\ell} \otimes \cI_{j,k}^{j+k} \\
& \cong & \cI_{i,j+k+\ell}^{i+j+k+\ell} \otimes \cI_{j+k,\ell}^{j+k+\ell} \otimes \cI_{j,k}^{j+k} 
& \cong &\cI_{i,j+k+\ell}^{i+j+k+\ell} \otimes \cI_{j,k,\ell}^{j+k+\ell} 
\end{array} \]
In particular, any element in one of the above spaces describes a composite of $\cT$-intertwining operators that takes elements of $M_i \otimes M_j \otimes M_k \otimes M_\ell$ and factors through a canonical embedding from $M_{i+j+k+\ell}[[z,w,t]][z^{-1}, w^{-1} t^{-1}, (z-w)^{-1}, (z-t)^{-1}, (w-t)^{-1}]$.
\end{lem}
\begin{proof}
The first isomorphism is the condition defining $\cI_{i,j,k,\ell}^{i+j+k+\ell}$.  The other isomorphisms follow from the conditions defining the set of spaces $\{ \cI_{i,j,k}^{i+j+k}\}_{i,j,k \in A}$.  The fact that all of the maps factor through $M_{i+j+k+\ell}[[z,w,t]][z^{-1}, w^{-1} t^{-1}, (z-w)^{-1}, (z-t)^{-1}, (w-t)^{-1}]$, and hence the corresponding commutative diagram of embeddings given in Lemma \ref{lem:star-of-embeddings}, follows from the standard fact that if $f: U \to V$ is a linear map of vector spaces that factors as a composite of linear maps $U \to W \to V$, and $g$ is a constant multiple of $f$, then $g$ factors through $W$.
\end{proof}

\begin{lem} \label{lem:pentagon}
Given a one-dimensional $\cT$-associativity datum, and any choice of nonzero intertwining operators $\{ m^{i,j}_z \in \cI_{i,j}^{i+j} \}_{i,j \in A}$, the function $F$ satisfies the pentagon identity: 
\[ F(i,j,k) F(i,j+k,\ell) F(j,k,\ell) = F(i+j,k,\ell) F(i,j,k+\ell) \]
for any $i,j,k,\ell \in A$.  In other words, the function $F$, viewed as a group cohomology 3-cochain, is in fact a 3-cocycle.
\end{lem}
\begin{proof}
For any $i,j,k,\ell \in A$, let $u_i \in M_i$, $u_j \in M_j$, $u_k \in M_k$ and $u_\ell \in M_\ell$ satisfy:
\[ m^{i,j+k+\ell}_z \circ (1 \otimes m^{j,k+\ell}_w) \circ (1 \otimes 1 \otimes m^{k,\ell}_t) (u_i \otimes u_j \otimes u_k \otimes u_\ell) \neq 0 \in M_{i+j+k+\ell}((z))((w))((t)). \]
For convenience, we write $M' = M_{i+j+k+\ell}[[z,w,t]][z^{-1}, w^{-1} t^{-1}, (z-w)^{-1}, (z-t)^{-1}, (w-t)^{-1}]$.  By the existence of $\cI^{i+j+k+\ell}_{i,j,k,\ell}$, a satisfactory choice of elements $(u_i,u_j,u_k,u_\ell)$ exists, and the map $m^{i,j+k+\ell}_z \circ (1 \otimes m^{j,k+\ell}_w) \circ (1 \otimes 1 \otimes m^{k,\ell}_t)$ is the composite of a map to $M'$ followed by the injection $\iota_{t,w,z}$.  We let $u'$ be the image of $u_i \otimes u_j \otimes u_k \otimes u_\ell$ in $M'$.

By Lemma \ref{lem:isoms}, we can follow the following diagram
\[ \xymatrix{ & & ((ij)k)\ell & & \\
(i(jk))\ell \ar[rru]^{F(i,j,k)} & & & & (ij)(k\ell) \ar[llu]_{F(i+j,k,\ell)} \\ \\
& i((jk)\ell) \ar[luu]^{F(i,j+k.\ell)} & & i(j(k\ell)) \ar[ll]_{F(j,k,\ell)} \ar[ruu]_{F(i,j,k+\ell)} & } \]
to deduce statements about the composition of other intertwining operators:
\begin{enumerate}
\item The composite $m^{i+j,k+\ell}_w \circ (1 \otimes m^{k,\ell}_t) \circ (m^{i,j}_{z-w} \otimes 1 \otimes 1)$ factors through $M'$, and takes $u_i \otimes u_j \otimes u_k \otimes u_\ell$ to $F(i,j,k+\ell) u'$.
\item The composite $m^{i,j+k+\ell}_z \circ (1 \otimes m^{j+k,\ell}_t) \circ (1 \otimes m^{j,k}_{w-t} \otimes 1)$ factors through $M'$, and takes $u_i \otimes u_j \otimes u_k \otimes u_\ell$ to $F(j,k,\ell) u'$.
\item The composite $m^{i+j+k,\ell}_t \circ (m^{i,j+k}_{z-t} \otimes 1) \circ (1 \otimes m^{j,k}_{w-t} \otimes 1)$ factors through $M'$, and takes $u_i \otimes u_j \otimes u_k \otimes u_\ell$ to $F(i,j+k,\ell) \cdot F(j,k,\ell) u'$.
\item The composite $m^{i+j+k,\ell}_t \circ (m^{i+j,k}_{w-t} \otimes 1) \circ (m^{i,j}_{z-w} \otimes 1 \otimes 1)$ factors through $M'$, and takes $u_i \otimes u_j \otimes u_k \otimes u_\ell$ to both $F(i,j,k) \cdot F(i,j+k,\ell) \cdot F(j,k,\ell) u'$ and $F(i+j,k,\ell) \cdot F(i,j,k+\ell) u'$.
\end{enumerate}
We find that $F(i,j,k) \cdot F(i,j+k,\ell) \cdot F(j,k,\ell) u' = F(i+j,k,\ell) \cdot F(i,j,k+\ell) u'$.  Since $u'$ is a nonzero element of a complex vector space, the result follows.
\end{proof}

\begin{defn}
Given a one-dimensional $\cT$-associativity datum, a choice $\{ m^{i,j}_z \in \cI_{i,j}^{i+j} \}_{i,j \in A}$ of nonzero intertwining operators is normalized if for all $i \in A$, $m^{0,i}_z = act_z$ and $m^{i,0}_z = act^*_z$.
\end{defn}

\begin{lem} \label{lem:normalized}
Suppose we are given a one-dimensional $\cT$-associativity datum.  Then, for any normalized choice $\{ m^{i,j}_z \in \cI_{i,j}^{i+j} \}_{i,j \in A}$ of nonzero intertwining operators, we have $F(0,i,j) = F(i,0,j) = F(i,j,0) = 1$ for all $i,j \in A$.  In other words, $F$ describes a normalized group cohomology 3-cocycle in the sense of Definition \ref{defn:normalized}.
\end{lem}
\begin{proof}
Fix elements $u \in M_i$ and $v \in M_j$ such that $m^{i,j}_w(u \otimes v) \neq 0$, and set $g(w) = \sum_{n\geq r} g_n w^n = m^{i,j}_w(u \otimes v) \in M_{i+j}((w))$.  By the defining property of $F(i,0,j)$, $m^{i,j}_z \circ (1 \otimes m^{0,j}_w)$ and $F(i,0,j) m^{i,j}_w \circ (m^{i,0}_{z-w} \otimes 1)$ are given by a single map from $M_i \otimes M_0 \otimes M_j$ to $M_{i,j}[[z,w]][z^{-1}, w^{-1}, (z-w)^{-1}]$, followed by the standard injections.  Applying this map to $u \otimes 1 \otimes v$, we find that $m^{i,j}_z (u \otimes v)$ and $F(i,0,j) m^{i,j}_w (e^{(z-w)L(-1)}u \otimes v)$ are expansions of the same element of $M_{i,j}[[z,w]][z^{-1}, w^{-1}, (z-w)^{-1}]$.  The left side is $g(z)$, while the right side is equal to $F(i,0,j) e^{(z-w)\partial_w} g(w)$ by Lemma \ref{lem:vafacts}.  By Lemma \ref{lem:taylor} applied to $M_{i,j}((w))[[z-w]]$, the right side is also equal to $F(i,0,j) g(z)$, so $F(i,0,j) = 1$.

Now, we use Lemma \ref{lem:pentagon}.  By substituting $k = \ell = 0$, we obtain $F(i,j,0) F(i,j,0) F(j,0,0) = F(i+j,0,0) F(i,j,0)$, and together with the relation $F(j,0,0) = F(i+j,0,0) = 1$ established in the previous paragraph, we get $F(i,j,0) = 1$.  By substituting $i=j=0$, we find that the relation $F(0,0,k) F(0,k,\ell) F(0,k,\ell) = F(0,k,\ell) F(0,0,k+\ell)$ together with the relation $F(0,0,k) = F(0,0,k+\ell) = 1$ established in the previous paragraph implies $F(0,k,\ell) = 1$.

\end{proof}


\subsection{Commutativity}

\begin{defn}
Let $V$ be a $\cT$-vertex algebra, let $A$ be an abelian group, and let $\{ M_i \}_{i \in A}$ be a set of $V$-modules in $\cT$, such that $M_0 = V$.  We define a one-dimensional $\cT$-commutativity datum to be a one-dimensional $\cT$-associativity datum that satisfies the additional condition that for any $i,j \in A$, and any nonzero $m^{i,j}_z \in \cI^{i+j}_{i,j}$, the intertwining operator $m^{i,j,*}_z$ is an element of $\cI^{j+i}_{j,i}$. In other words, for any nonzero $m^{j,i}_z \in \cI^{j+i}_{j,i}$, there exists $\lambda \in \bC^\times$ such that for any $u \otimes v \in M_i \otimes M_j$, the elements $\lambda m^{i,j}_z(u \otimes v)$ and $e^{zL(-1)}m^{j,i}_{-z}(v \otimes u)$ of $M_{i+j}((z))$ are equal.
\end{defn}

\begin{rem}
If $\cT$-vertex algebras are vertex operator algebras, then a commutativity datum is quite similar to Huang's notion of ``$\bZ$-graded meromorphic locally grading-restricted conformal intertwining algebra'' \cite{H04}, where the set $\cA$ is our group $A$, and the vector spaces $\cV_{i,j}^k$ are $\cI_{i,j}^{i+j}$ when $k=i+j$ and $0$ otherwise.  However, Huang requires a convergence condition for all compositions of intertwining operators, while we only require three-fold compositions to be well-behaved.  The reader should be aware that although the word ``algebra'' appears in the name, Huang's notion of intertwining algebra does not have a distinguished choice of multiplication operator in the defining data.
\end{rem}

\begin{lem} \label{lem:comm-span}
Suppose we are given a one-dimensional $\cT$-commutativity datum.  Then for any choice of nonzero intertwining operators $\{ m^{i,j}_z \in \cI_{i,j}^{i+j} \}_{i,j \in A}$, and any $i,j,k \in A$, the compositions $m^{i+j,k}_w \circ (m^{i,j}_{z-w} \otimes 1)$ and $m^{i+j,k}_z \circ (m^{j,i}_{w-z} \otimes 1) \circ \tau_{(12)}$, factor as maps $M_i \otimes M_j \otimes M_k \to M_{i+j+k}[[z,w]][z^{-1},w^{-1},(z-w)^{-1}]$ that differ by a constant multiple that does not depend on $k$, followed by inclusions $\iota_{w,z-w}$ and $\iota_{z,z-w}$.  In particular, the vector spaces $\tau_{(12)}^*(\cI^{i+j+k}_{i,j,k})$ and $\cI^{j+i+k}_{j,i,k}$ are equal.
\end{lem}
\begin{proof}
Let $u \otimes v \otimes x \in M_i \otimes M_j \otimes M_k$ satisfy $m^{i+j,k}_w (m^{i,j}_{z-w} (u \otimes v) \otimes x) \neq 0 \in M_{i+j+k}((w))((z-w))$.  Such an element exists, by the hypothesis that $\cI^{i+j+k}_{i,j,k}$ is nonzero.  By the condition defining commutativity datum, $\lambda m^{i,j}_{z-w}(u \otimes v) = e^{(z-w)L(-1)} m^{j,i}_{w-z}(v \otimes u)$ for some $\lambda \in \bC^\times$ that depends only on $i$ and $j$.  By the translation covariance axiom applied to $m^{i+j,k}_w$, we have
\[ m^{i+j,k}_w( (e^{(z-w)L(-1)} m^{j,i}_{w-z}(v \otimes u)) \otimes x) = e^{(z-w)\partial_w} m^{i+j,k}_w(m^{j,i}_{w-z}(v \otimes u) \otimes x) \]
as elements of $M_{i+j+k}((w))((z-w))$.  However, these series are the image of a unique element of $M_{i+j+k}[[z,w]][z^{-1}, w^{-1}, (z-w)^{-1}]$ along the embedding given in Lemma \ref{lem:basic-embeddings}.  Thus, by Lemma \ref{lem:weird-isom}, the expression $m^{i+j,k}_z (m^{j,i}_{w-z}(v \otimes u) \otimes x)$ is the expansion in $M_{i+j+k}((z))((z-w))$ of the same element.  Assembling these equalities, we find that $\lambda m^{i+j,k}_w \circ (m^{i,j}_{z-w}(u \otimes v) \otimes x)$ and $m^{i+j,k}_z (m^{j,i}_{w-z}(v \otimes u) \otimes x)$ are expansions of the same nonzero element of $M_{i+j+k}[[z,w]][z^{-1}, w^{-1}, (z-w)^{-1}]$.
\end{proof}

\begin{defn}
Given a one dimensional $\cT$-commutativity datum, for any choice of nonzero operators $\{ m^{i,j}_z \in \cI_{i,j}^{i+j} \}_{i,j \in A}$ we define the function $\Omega: A^{\oplus 2} \to \bC^\times$ by
\[ m^{i,j}_z = \Omega(i,j) e^{zL(-1)}m^{j,i}_{-z}\circ \tau_{(12)} \].
\end{defn}

\begin{rem}
We note that by Lemma \ref{lem:comm-span}, the following holds for all $k \in A$:
\[ m^{i+j,k}_w \circ (m^{i,j}_{z-w} \otimes 1) = \Omega(i,j) m^{j+i,k}_z \circ (m^{j,i}_{-z+w} \otimes 1) \circ \tau_{(12)}. \]
\end{rem}

\begin{rem}
One may alternatively define $\Omega$ by setting $B(i,j,k)$ to be the defect in the original locality diagram, and setting $\Omega(i,j) = F(i,j,k)^{-1} B(i,j,k) F(j,i,k)$ (see \cite{DL93} Chapter 12).  In terms of intertwining operators, we see that the two composite maps
\[ m^{i,j+k}_z \circ (1 \otimes m^{j,k}_w) \quad \text{and} \quad F(j,i,k)^{-1} \Omega(i,j) F(i,j,k) m^{j,i+k}_w \circ (1 \otimes m^{i,k}_z) \circ \tau_{(12)} \]
 factor through equal maps from $M_i \otimes M_j \otimes M_k$ to $M_{i+j+k}[[z,w]][z^{-1},w^{-1},(z-w)^{-1}]$.  Diagrammatically, $\Omega(i,j)$ can be represented as the defect obstructing commutativity in the outer pentagon:

\[ \xymatrix{
(M_i \otimes M_j) \otimes M_k \ar[r]^{F(i,j,k)^{-1}} & M_i \otimes (M_j \otimes M_k) \ar[rr]^{B(i,j,k)} & & M_j \otimes (M_i \otimes M_k) \ar[r]^{F(j,i,k)} & (M_j \otimes M_i) \otimes M_k \\
& & {}\save[]*{M_{i+j+k}[[z,w]][z^{-1},w^{-1},(z-w)^{-1}]} \ar@{<-}[ul] \ar@{<-}[ull] \ar@{<-}[ur] \ar@{<-}[urr] \restore
} \] 
\end{rem}

\begin{lem} \label{lem:hexagon}
Given a one dimensional $\cT$-commutativity datum, for any choice of nonzero elements $\{ m^{i,j}_z \in \cI_{i,j}^{i+j} \}_{i,j \in A}$, the functions $F$ and $\Omega$ satisfy the hexagon conditions:
\begin{enumerate}
\item $F(i,j,k)^{-1}\Omega(i,j+k) F(j,k,i)^{-1} = \Omega(i,j) F(j,i,k)^{-1} \Omega(i,k)$
\item $F(i,j,k) \Omega(i+j,k) F(k,i,j) = \Omega(j,k) F(i,k,j) \Omega(i,k)$
\end{enumerate}
\end{lem}
\begin{proof}
As in Lemma \ref{lem:pentagon}, for any $i,j,k \in A$, let $u_i \in M_i$, $u_j \in M_j$, $u_k \in M_k$ and $u_0 \in M_0 = V$ satisfy:
\[ m^{i,j+k}_z \circ (1 \otimes m^{j,k}_w) \circ (1 \otimes 1 \otimes m^{k,0}_t) (u_i \otimes u_j \otimes u_k \otimes u_0) \neq 0 \in M_{i+j+k}((z))((w))((t)). \]
For convenience, we write $M' = M_{i+j+k}[[z,w,t]][z^{-1}, w^{-1} t^{-1}, (z-w)^{-1}, (z-t)^{-1}, (w-t)^{-1}]$.  As before, by the existence of $\cI^{i+j+k}_{i,j,k,0}$, a satisfactory choice of elements $(u_i,u_j,u_k,u_0)$ exists, and the map $m^{i,j+k}_z \circ (1 \otimes m^{j,k}_w) \circ (1 \otimes 1 \otimes m^{k,0}_t)$ is the composite of a map to $M'$ followed by the injection $\iota_{t,w,z}$.  We let $u'$ be the image of $u_i \otimes u_j \otimes u_k \otimes u_0$ in $M'$ in this factorization.  

To prove the first equation, we follow the following diagram:
\[ \xymatrix{ & j((ki)0) & j((ik)0) \ar[l]_{\Omega(i,k)} & \\
(j(ki))0 \ar[ur]^{F(j,i+k,0)^{-1} = 1} & & & (j(ik))0 \ar[ul]_{F(j,i+k,0)^{-1} = 1} \\
((jk)i)0 \ar[u]^{F(j,k,i)^{-1}} & & & ((ji)k)0 \ar[u]_{F(j,i,k)^{-1}} \\
& (i(jk))0 \ar[ul]^{\Omega(i,j+k)} & ((ij)k)0 \ar[l]_{F(i,j,k)^{-1}} \ar[ur]_{\Omega(i,j)} } \]
to deduce statements about the composition of the corresponding intertwining operators.  The left side of the octagon is traversed as follows:
\begin{enumerate}
\item[((ij)k)0:] By Lemma \ref{lem:pentagon}, the composite $m^{i+j+k,0}_t \circ (m^{i+j,k}_{w-t} \otimes 1) \circ (m^{i,j}_{z-w} \otimes 1 \otimes 1)$ factors through $M'$, and takes $u_i \otimes u_j \otimes u_k \otimes u_0$ to $F(i,j,k)u'$.
\item[(i(jk))0:] The composite $m^{i+j+k,0}_t \circ (m^{i,j+k}_{z-t} \otimes 1) \circ (1 \otimes m^{j,k}_{w-t} \otimes 1)$ factors through $M'$, and takes $u_i \otimes u_j \otimes u_k \otimes u_0$ to $u'$.
\item[((jk)i)0:] By Lemma \ref{lem:comm-span}, the composite $m^{j+k+i,0}_z \circ (m^{j+k,i}_{t-z} \otimes 1) \circ (m^{j,k}_{w-t} \otimes 1 \otimes 1)$ factors through $M'$ and takes $u_i \otimes u_j \otimes u_k \otimes u_0$ to $\Omega(i,j+k) u'$.
\item[(j(ki))0:] The composite $m^{j+k+i,0}_t \circ (m^{j,k+i}_{w-z} \otimes 1) \circ (1 \otimes m^{k,i}_{t-z} \otimes 1)$ factors through $M'$, and takes $u_i \otimes u_j \otimes u_k \otimes u_0$ to $F(j,k,i)^{-1} \Omega(i,j+k) u'$.
\item[j((ki)0):] By Lemma \ref{lem:normalized}, the composite $m^{j,k+i}_{w-z} \circ (1 \otimes m^{k+i,0}_z) \circ (1 \otimes m^{k,i}_{t-z} \otimes 1)$ factors through $M'$, and takes $u_i \otimes u_j \otimes u_k \otimes u_0$ to $F(j,i+k,0)^{-1} F(j,k,i)^{-1} \Omega(i,j+k) u' = F(j,k,i)^{-1} \Omega(i,j+k) u'$.
\end{enumerate}
The right side of the octagon is traversed as follows:
\begin{enumerate}
\item[((ji)k)0:] The composite $m^{j+i+k,0}_t \circ (m^{j+i,k}_{z-t} \otimes 1) \circ (m^{j,i}_{w-z} \otimes 1 \otimes 1)$ factors through $M'$, and takes $u_i \otimes u_j \otimes u_k \otimes u_0$ to $\Omega(i,j) F(i,j,k)u'$.
\item[(j(ik))0:] The composite $m^{j+i+k,0}_t \circ (m^{j,i+k}_{w-t} \otimes 1) \circ (1 \otimes m^{i,k}_{z-t} \otimes 1)$ factors through $M'$, and takes $u_i \otimes u_j \otimes u_k \otimes u_0$ to $F(j,i,k)^{-1} \Omega(i,j) F(i,j,k) u'$.
\item[j((ik)0):] The composite $m^{j,i+k}_{w-t} \circ (1 \otimes m^{i+k,0}_t) \circ (1 \otimes m^{i,k}_{z-t} \otimes 1)$ factors through $M'$, and takes $u_i \otimes u_j \otimes u_k \otimes u_0$ to $F(j,i+k,0)^{-1} F(j,i,k)^{-1} \Omega(i,j) F(i,j,k) u' = F(j,i,k)^{-1} \Omega(i,j) F(i,j,k) u'$.
\item[j((ki)0):] The composite $m^{j,k+i}_{w-z} \circ (1 \otimes m^{k+i,0}_z) \circ (1 \otimes m^{k,i}_{t-z} \otimes 1)$ factors through $M'$, and takes $u_i \otimes u_j \otimes u_k \otimes u_0$ to $\Omega(i,k) F(j,i,k)^{-1} \Omega(i,j) F(i,j,k) u'$
\end{enumerate}
We see that $F(j,k,i)^{-1} \Omega(i,j+k) u' = \Omega(i,k) F(j,i,k)^{-1} \Omega(i,j) F(i,j,k) u'$, and because $u'$ is a nonzero element of a complex vector space, the first hexagon equation holds.

To prove the second equation, we follow the following diagram:
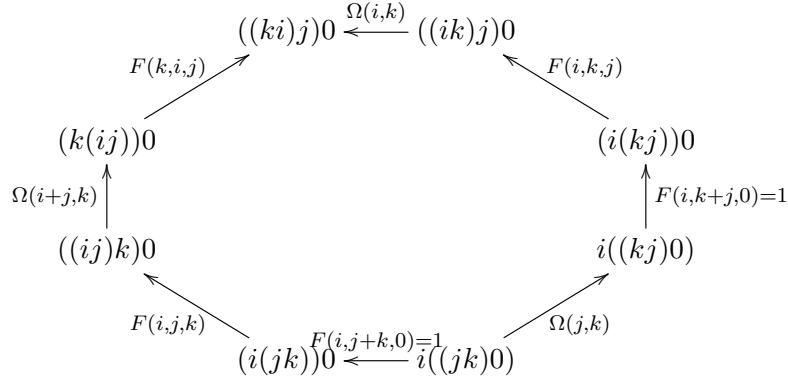
\[ \xymatrix{ & ((ki)j)0 & ((ik)j)0 \ar[l]_{\Omega(i,k)} & \\
(k(ij))0 \ar[ur]^{F(k,i,j)} & & & (i(kj))0 \ar[ul]_{F(i,k,j)} \\
((ij)k)0 \ar[u]^{\Omega(i+j,k)} & & & i((kj)0) \ar[u]_{F(i,k+j,0) = 1} \\
& (i(jk))0 \ar[ul]^{F(i,j,k)} & i((jk)0) \ar[l]_{F(i,j+k,0) = 1} \ar[ur]_{\Omega(j,k)} } \]
to deduce statements about the composition of the corresponding intertwining operators.  After omitting previous evaluations, the left side of the octagon is traversed as follows:
\begin{enumerate}
\item[((ij)k)0:] The composite $m^{i+j+k,0}_t \circ (m^{i+j,k}_{w-t} \otimes 1) \circ (m^{i,j}_{z-w} \otimes 1 \otimes 1)$ factors through $M'$ and takes $u_i \otimes u_j \otimes u_k \otimes u_0$ to $F(i,j,k) u'$.
\item[(k(ij))0:] The composite $m^{k+i+j,0}_w \circ (m^{k,i+j}_{t-w} \otimes 1) \circ (1 \otimes m^{i,j}_{z-w} \otimes 1)$ factors through $M'$, and takes $u_i \otimes u_j \otimes u_k \otimes u_0$ to $\Omega(i+j,k) F(i,j,k) u'$.
\item[((ki)j)0:] The composite $m^{k+i+j,0}_w \circ (m^{k+i,j}_{z-w} \otimes 1) \circ (m^{k,i}_{t-z} \otimes 1 \otimes 1)$ factors through $M'$ and takes $u_i \otimes u_j \otimes u_k \otimes u_0$ to $F(k,i,j)\Omega(i+j,k) F(i,j,k) u'$.
\end{enumerate}
The right side of the octagon is traversed as follows:
\begin{enumerate}
\item[i((kj)0):] The composite $m^{i,k+j}_{z-w} \circ (1 \otimes m^{k+j,0}_w) \circ (1 \otimes m^{k,j}_{t-w} \otimes 1)$ factors through $M'$, and takes $u_i \otimes u_j \otimes u_k \otimes u_0$ to $\Omega(j,k) u'$.
\item[(i(kj))0:] The composite $m^{i+k+j,0}_w \circ (m^{i,k+j}_{z-w} \otimes 1) \circ (1 \otimes m^{k,j}_{t-w} \otimes 1)$ factors through $M'$, and takes $u_i \otimes u_j \otimes u_k \otimes u_0$ to $F(i,k+j,0)\Omega(j,k) u' = \Omega(j,k) u'$.
\item[((ik)j)0:] The composite $m^{i+k+j,0}_w \circ (m^{i+k,j}_{t-w} \otimes 1) \circ (m^{i,k}_{z-t} \otimes 1 \otimes 1)$ factors through $M'$ and takes $u_i \otimes u_j \otimes u_k \otimes u_0$ to $F(i,k,j) \Omega(j,k) u'$.
\item[((ki)j)0:] The composite $m^{k+i+j,0}_w \circ (m^{k+i,j}_{z-w} \otimes 1) \circ (m^{k,i}_{t-z} \otimes 1 \otimes 1)$ factors through $M'$ and takes $u_i \otimes u_j \otimes u_k \otimes u_0$ to $\Omega(i,k)F(i,k,j) \Omega(j,k) u'$.
\end{enumerate}
We see that $F(k,i,j)\Omega(i+j,k) F(i,j,k) u' = \Omega(i,k)F(i,k,j) \Omega(j,k) u'$, and because $u'$ is a nonzero element of a complex vector space, the second hexagon equation holds.
\end{proof}

\begin{prop} \label{prop:get-cocycle}
Let $V$ be a $\cT$-vertex algebra, let $A$ be an abelian group, and let $\{ M_i \}_{i \in A}$ be a set of $V$-modules in $\cT$, such that $M_0 = V$.  Given a one dimensional $\cT$-commutativity datum, and any normalized choice of nonzero elements $\{ m^{i,j}_z \in \cI_{i,j}^{i+j} \}_{i,j \in A}$, the pair $(F, \Omega)$ derived from $\{ m^{i,j} \}$ forms a normalized abelian 3-cocycle for $A$ with coefficients in $\bC^\times$.
\end{prop}
\begin{proof}
This follows from the fact that the abelian 3-cocycle condition is given by the pentagon condition (proved in Lemma \ref{lem:pentagon}), and the hexagon condition (proved in Lemma \ref{lem:hexagon}).
\end{proof}

\begin{rem}
Before we can state our main obstruction results, we have to deal with the possibility of finiteness conditions obstructing the formation of $\cT$-vertex algebras.  In particular, vertex operator algebras satisfy the condition that eigenvalues of $L(0)$ have finite multiplicity, and this may not be satisfied when $A$ is infinite.  The standard example is given by vertex algebras of indefinite lattices, such as the fake monster vertex algebra (introduced in \cite{B86} without a name), which are built from modules of a Heisenberg vertex operator algebra, but are not vertex operator algebras because of this spectral condition.
\end{rem}

\begin{defn} \label{defn:T-compatible}
We say that a collection of modules in $\cT$ is $\cT$-compatible if their direct sum is a module in $\cT$.
\end{defn}

\begin{prop} \label{prop:criterion-for-existence}
Let $V$ be a $\cT$-vertex algebra, let $A$ be an abelian group, and let $\{ M_i \}_{i \in A}$ be a $\cT$-compatible set of $V$-modules in $\cT$, such that $M_0 = V$.  Suppose we are given a one dimensional $\cT$-commutativity datum, and a normalized choice of nonzero elements $\{ m^{i,j}_z \in \cI_{i,j}^{i+j} \}_{i,j \in A}$.  Then we define a vector space $W = \bigoplus_{i \in A} M_i$, a derivation $L(-1)^W: W \to W$ by its restriction to each $M_i$, and a multiplication map $m^W_z: W \otimes W \to W((z))$ as the sum of its restrictions $m^{i,j}_z$.  The tuple $(W, \unit_V, L(-1)^W,m^W_z)$ is a $\cT$-vertex algebra if and only if the abelian 3-cocycle $(F, \Omega)$ is identically 1.
\end{prop}
\begin{proof}
Since $F$ measures the failure of associativity, the associativity property holds if and only if $F$ is identically 1.  $\Omega$ is identically 1 if and only if skew-symmetry holds.  Since locality is equivalent to the combination of associativity and skew-symmetry, we obtain a vertex algebra.  We now need to show that we obtain a $\cT$-vertex algebra for the various values of $\cT$ under consideration.
\begin{enumerate}
\item The unrestricted setting: We are already done.
\item Worldsheet symmetry: If $V$ is weighted (resp. M\"obius, quasi-conformal), our modules come with a collection of operators $\{L(i)\}$, and the compatibility condition for module structures and intertwining operators is equivalent for the vertex algebra to have the new worldsheet symmetry.  If $V$ is conformal, the conformal condition on modules implies the conformal vector induces a Virasoro action on $W$.
\item VOAs: Our assumption that $\{ M_i \}_{i \in A}$ is $\cT$-compatible implies $W$ has spectrum bounded below and finite dimensional eigenspaces.
\item $G$-equivariant versions: The compatibility condition implies $g m_z^{i,j}(u_i \otimes u_j) = m_z^{i,j}(g u_i \otimes g u_j)$ for all $u_i \in M_i$ and $u_j \in M_j$.  This is precisely the condition for $g$ to be a vertex algebra automorphism of $W$.  For the action of $G$ to respect additional worldsheet symmetry on $W$, it is necessary and sufficient that $g L(k) = L(k) g$ for all relevant $k$, and this one of the conditions for a $G$-action on $M_i$ to be compatible with the $V$-module structure.
\end{enumerate}
\end{proof}

\begin{defn} \label{defn:twist-action}
Let $V$ be a $\cT$-vertex algebra, let $A$ be an abelian group, and let $\{ M_i \}_{i \in A}$ be a set of $V$-modules in $\cT$, such that $M_0 = V$.  Suppose we are given spaces $\cI_{i,j}^{i+j}$ of $\cT$-intertwining operators  $\{ m^{i,j}_z: M_i \otimes M_j \to M_{i+j}((z)) \}_{i,j \in A}$.  We define a ``$\lambda$-twist'' action of the abelian 2-cochain group $C^2_{ab}(A,\bC^\times)$ on $\bigoplus \cI_{i,j}^{i+j}$ by sending any cochain $(i,j) \mapsto \lambda_{i,j} \in \bC^\times$ to the rescaling map $\{ m^{i,j}(z) \} \mapsto \{ \lambda_{i,j} m^{i,j}(z) \}$.
\end{defn}

\begin{thm} \label{thm:obstruction}
Let $V$ be a $\cT$-vertex algebra, let $A$ be an abelian group, and let $\{ M_i \}_{i \in A}$ be a set of $V$-modules in $\cT$, such that $M_0 = V$.  Given a one dimensional $\cT$-commutativity datum attached to $V$ and $\{M_i\}_{i \in A}$, and any normalized choice of nonzero elements $\{ m^{i,j}_z \in \cI_{i,j}^{i+j} \}_{i,j \in A}$, the following hold:
\begin{enumerate}
\item The action of $C^2_{ab}(A,\bC^\times)$ on the total intertwining operator space $\bigoplus \cI_{i,j}^{i+j}$ given in Definition \ref{defn:twist-action} induces the canonical translation action $(\{\lambda_{i,j}\}, (F, \Omega))\mapsto (d\lambda \cdot (F, \Omega))$ on the group of abelian 3-cocycles by coboundaries.  This action is transitive on representatives of any fixed cohomology class in $H^3_{ab}(A,\bC^\times)$, with stabilizer given by the group of abelian 2-cocycles $\{\lambda_{i,j} \}$.
\item The function $A \to \pm 1$ defined by $i \mapsto \Omega(i,i)$ is a quadratic form invariant under $\lambda$-twist by 2-cochains.  In particular, the abelian cohomology class of $(F, \Omega)$ is canonically attached to the commutativity datum.
\item If $\{M_i\}$ is $\cT$-compatible, then the quadratic form $i \mapsto \Omega(i,i)$ is identically one if and only if there exists a normalized 2-cochain $\lambda$ such that $\{ \lambda_{i,j} m^{i,j}_z \}$ describe an $A$-graded $\cT$-vertex algebra structure on $\bigoplus_{i \in A} M_i$.
\item The $A$-graded $\cT$-vertex algebra structure on $\bigoplus_{i \in A} M_i$ is unique up to isomorphism if it exists.
\end{enumerate}
\end{thm}
\begin{proof}
For the first claim, it suffices to verify that the cocycle $(F, \Omega)$ is shifted by the abelian differential of $\lambda$.  This is essentially the content of \cite{DL93} Remark 12.23 (for finite order cocycles), and \cite{BK06} section 3.4 (for the case of trivial $F$), and the arguments given in the literature extend without substantial change.

In the second statement, the invariance of the abelian cohomology class follows from the fact that normalized 2-cochains act simply transitively on the set of normalized choices of intertwining operators.  Both the fact that $i \to \Omega(i,i)$ is quadratic and the fact that it is invariant follow from Lemma \ref{lem:cocycle-trace}.

For the third statement, Proposition \ref{prop:criterion-for-existence} implies it is necessary and sufficient that we obtain a trivial abelian 3-cocycle.  By Lemma \ref{lem:normalized-differential}, the abelian differential defines a transitive action of normalized 2-cochains on the normalized abelian 3-cocycles in the abelian cohomology class of $(F, \Omega)$.  In other words there exists such $\lambda$ if and only if the pair $(F, \Omega)$ is an abelian 3-cocycle in the zero cohomology class.  By Lemma \ref{lem:cocycle-trace}, the abelian cohomology class is determined by its trace, and is trivial if and only if $\Omega(i,i) = 1$ for all $i$.

For the last claim, suppose a $\cT$-vertex algebra compatible with the commutativity datum exists.  Then a $\lambda$-twist yields a $\cT$-vertex algebra precisely when $\lambda$ is an abelian 2-cocycle, i.e., we have an action of $Z^2_{ab}(A,\bC^\times) = Z^3(K(A,2),\bC^\times)$ on the set of isomorphism classes of $A$-graded $\cT$-vertex algebras compatible with the commutativity datum. Since $\lambda$-twisting is transitive on the set of choices of multiplication rules, the cocycle group acts transitively on the set of isomorphism classes.  Furthermore, this action factors through $H^2_{ab}(A,\bC^\times)$, because twists by 2-coboundaries induce isomorphisms by abelian 1-cochains given by multiples of identity $c_i \cdot Id: M_i \to M_i$.  By Lemma \ref{lem:degree-2-abelian-cohomology}, $H^2_{ab}(A,\bC^\times)$ is trivial.  Thus, the set of isomorphism classes has a transitive action of a trivial group, and hence has one element.
\end{proof}

\begin{rem}
The uniqueness claim is similar to the argument in the proof of Proposition 5.3 in \cite{DM02}.  However, the argument there needs some minor repair: all instances of ``2-cocycle'' and ``two cocycle'' should be changed to ``abelian 2-cocycle''.  Indeed, $H^2(A,\bC^\times)$ is nontrivial for any non-cyclic finite abelian group $A$.
\end{rem}

\subsection{Evenness}

The condition that $\Omega(i,i) = 1$ for all $i \in A$ that appears in the third claim of Theorem \ref{thm:obstruction} is not automatically satisfied.  Indeed, any supercommutative ring $R$ whose odd part has nonzero three-fold products can be written as follows: We set $A = \bZ/2\bZ$, $V = M_0 = R_{even}$, $M_1 = R_{odd}$, $L(-1)$ is the zero map, and $m^{i,j}$ are the restrictions of multiplication on $R$.  We find that the spans of $\{ m^{i,j} \}$ form a one-dimensional $\cT$-commutativity datum, but the abelian 3-cocycle is nontrivial: $F$ is the zero map, but $\Omega(i,j) = (-1)^{ij}$.  This counterexample works in the weighted, M\"obius, quasi-conformal, and conformal cases, because we may define a conformal element by setting $\omega = 0$, causing $L(i)$ to act by zero for all $i \in \bZ$.  We explore this condition further here.

\begin{defn}
Let $V$ be a $\cT$-vertex algebra, and let $M_1$ and $M_2$ be $V$-modules in $\cT$.  A $\cT$-intertwining operator $I_z: M_1 \otimes M_1 \to M_2((z))$ is called even if one of the following conditions holds:
\begin{enumerate}
\item For any $u \in M_1$, $I_z(u \otimes u) = 0$, and for any $u, v \in M_1$ such that $I_z(u \otimes v) \neq 0$, we have $I_z(u \otimes v) \in z^n M_2[[z]]\setminus z^{n+1}M_2[[z]]$, where $n$ is odd.
\item There exists $u \in M_1$ such that $I_z(u \otimes u) \in z^n M_2[[z]]\setminus z^{n+1}M_2[[z]]$, where $n$ is even.
\end{enumerate}
$I_z$ is called odd if one of the following conditions holds:
\begin{enumerate}
\item For any $u \in M_1$, $I_z(u \otimes u) = 0$, and for any $u, v \in M_1$ such that $I_z(u \otimes v) \neq 0$, we have $I_z(u \otimes v) \in z^n M_2[[z]]\setminus z^{n+1}M_2[[z]]$, where $n$ is even.
\item There exists $u \in M_1$ such that $I_z(u \otimes u) \in z^n M_2[[z]]\setminus z^{n+1}M_2[[z]]$, where $n$ is odd.
\end{enumerate}
A commutativity datum is even if for all $i \in A$, any nonzero $m^{i,i}_z \in \cI_{i,i}^{2i}$ is even.
\end{defn}

\begin{rem}
By Proposition 11.9 in \cite{DL93}, if $M_1$ is irreducible, then $I_z(u \otimes u) \neq 0$ for any nonzero $u \in M_1$ and any nonzero intertwining operator $I_z: M_1 \otimes M_1 \to M_2((z))$.  In this case, the term ``evenness'' is reasonable, since it refers to the exponent of the lowest order nonzero term in $I_z(u \otimes u)$.
\end{rem}

\begin{lem} \label{lem:evenness}
Let $V$ be a $\cT$-vertex algebra, let $M_1$ and $M_2$ be $V$-modules in $\cT$, and let $I_z: M_1 \otimes M_1 \to M_2((z))$ be a nonzero $\cT$-intertwining operator.  If $I_z$ is proportional to $I_z^*$, then $I_z$ is either even or odd, and rescaling an intertwining operator does not change evenness or oddness.  In particular, if we are given a one dimensional $\cT$-commutativity datum, for any $i \in A$, any nonzero $m^{i,i} \in \cI_{i,i}^{2i}$ is either even or odd.  Furthermore, for all $i \in A$, $m^{i,i}_z$ is even if and only if $\Omega(i,i) = 1$, and $m^{i,i}_z$ is odd if and only if $\Omega(i,i) = -1$.  In particular, the commutativity datum is even if and only if $\Omega(i,i) = 1$ for all $i \in A$.
\end{lem}
\begin{proof}
We apply the technique used in the proof of Lemma \ref{lem:intertwining-symmetrizer}.

Alternating case: If $I_z(u \otimes u) = 0$ for all $u \in M_1$, then bilinearity implies that for any $u, v \in M_1$, $I_z(u \otimes v) = -I_z(v \otimes u)$.  Since $I_z$ is nonzero, there exist $u,v \in M_1$ such that $I_z(u \otimes v) \in z^n M_2[[z]]\setminus z^{n+1}M_2[[z]]$ for some integer $n$, so we write $I_z(u \otimes v) = g_n z^n + g_{n+1} z^{n+1} + \cdots$ with $g_n$ nonzero.  Then by our assumption about $I_z^*$:
\[ \begin{aligned}
g_n z^n + \cdots &= I_z(u \otimes v) \\
&= \Omega(1,1) e^{zL(-1)}I^*_{-z}(v \otimes u) \\
&= -\Omega(1,1) e^{zL(-1)} g_n (-z)^n + \cdots \\
&= (-1)^{n+1} \Omega(1,1) g_n z^n + \cdots
\end{aligned}\]
If $n$ is odd (i.e., if $I_z$ is even), then $g_n = \Omega(1,1)g_n$, and $\Omega(1,1) = 1$.  If $n$ is even (i.e., if $I_z$ is odd), then $g_n = -\Omega(1,1)g_n$, and $\Omega(1,1) = -1$.

Non-alternating case: If there exists $u \in M_1$ such that $I_z(u \otimes u) \neq 0$, then $I_z(u \otimes u) \in z^n M_2[[z]]\setminus z^{n+1}M_2[[z]]$ for some integer $n$.  We write $I_z(u \otimes u) = g_n z^n + g_{n+1} z^{n+1} + \cdots$ with $g_n$ nonzero.  Then by our assumption about $I_z^*$:
\[ \begin{aligned}
g_n z^n + \cdots &= I_z(u \otimes u) \\
&= \Omega(1,1) e^{zL(-1)}I^*_{-z}(u \otimes u) \\
&= \Omega(1,1) e^{zL(-1)} g_n (-z)^n + \cdots \\
&= (-1)^n \Omega(1,1) g_n z^n + \cdots
\end{aligned}\]
If $n$ (hence $I_z$) is even, then $g_n = \Omega(1,1)g_n$, and $\Omega(1,1) = 1$.  If $n$ (hence $I_z$) is odd, then $g_n = -\Omega(1,1)g_n$, and $\Omega(1,1) = -1$.

The corresponding statements for commutativity data follow in a straightforward way.
\end{proof}

\begin{lem}
Let $V' \subset V$ be an embedding of vertex algebras, let $M_1$ and $M_2$ be $V$-modules in $\cT$, and let $I_z: M_1 \otimes M_1 \to M_2((z))$ be an intertwining operator in $\cT$ proportional to $I_z^*$.  If there is a $V'$-submodule $M'_1 \subset M_1$ in $\cT$ such that the restriction of $I_z$ to $M'_1 \otimes M'_1$ is nonzero, then $I_z$ is even if and only if its restriction is.
\end{lem}
\begin{proof}
By Lemma \ref{lem:evenness}, $I_z$ is either even or odd, and the same is true of the restriction.  Then evenness can be tested on a single pair of vectors.
\end{proof}

\begin{lem}
Let $V$ and $V'$ be $\cT$-vertex algebras, let $I_z: M_1 \otimes M_1 \to M_2((z))$ (resp. $I'_z: M'_1 \otimes M'_1 \to M'_2((z))$) be a $\cT$-intertwining operator of $V$-modules (resp. $V'$-modules) that is proportional to $I_z^*$ (resp. $I_z^{\prime,*}$).  Then $I_z \otimes I'_z: (M_1 \otimes M'_1) \otimes (M_1 \otimes M'_1) \to (M_2 \otimes M'_2)((z))$ is even if and only if either $I_z$ and $I'_z$ are both even or both odd.  If exactly one of $I_z$ and $I'_z$ is even, then $I_z \otimes I'_z$ is odd.
\end{lem}
\begin{proof}
This follows straightforwardly from the definition.
\end{proof}

\begin{prop} \label{prop:existence}
Let $V$ be a $\cT$-vertex algebra, let $A$ be an abelian group, and let $\{ M_i \}_{i \in A}$ be a $\cT$-compatible set of $V$-modules in $\cT$, such that $M_0 = V$.  Given a one dimensional even $\cT$-commutativity datum, there exists a choice of nonzero elements $\{ m^{i,j}_z \in \cI_{i,j}^{i+j} \}_{i,j \in A}$ that defines a $\cT$-vertex algebra structure on $\bigoplus_{i \in A} M_i$.  Furthermore, the vertex algebra structure is unique up to isomorphism.
\end{prop}
\begin{proof}
By Theorem \ref{thm:obstruction}, to prove the first assertion, it suffices to show that $\Omega(i,i) = 1$ for all $i \in A$.  This follows from Lemma \ref{lem:evenness}, because the commutativity datum is even.  The uniqueness follows from the last claim in Theorem \ref{thm:obstruction}.
\end{proof}

\begin{cor} \label{cor:2A-is-even}
Let $V$ be a $\cT$-vertex algebra, let $A$ be an abelian group, let $2\cdot A$ denote the subgroup of $A$ whose elements are even multiples, and let $\{ M_i \}_{i \in A}$ be a set of $V$-modules in $\cT$, such that $M_0 = V$, and $\{ M_i \}_{i \in 2A}$ is $\cT$-compatible.  Given a one dimensional $\cT$-commutativity datum, there exists a choice of nonzero elements $\{ m^{i,j}_z \in \cI_{i,j}^{i+j} \}_{i,j \in 2 \cdot A}$ that defines a $\cT$-vertex algebra structure on $\bigoplus_{i \in 2\cdot A} M_i$.  In particular, evenness is automatic for 2-divisible groups, such as finite abelian groups of odd order.
\end{cor}
\begin{proof}
By \ref{lem:evenness}, the function $i \mapsto \Omega(i,i)$ takes values in $\pm 1$, and by Lemma \ref{lem:cocycle-trace}, it is a quadratic form, so by Lemma \ref{lem:forms-with-restricted-values}, $i \mapsto \Omega(2i, 2i)$ is identically 1.  Then the result follows from Proposition \ref{prop:existence}.
\end{proof}

\begin{cor}
Given a $\cT$-commutativity datum for a $\cT$-compatible set of modules $\{M_i\}_{i \in A}$, the sum $\bigoplus_{i \in A} M_i$ admits a vertex algebra structure induced by the intertwining operators if and only if for each coset $2A+k$ in $A$, there exists some representative $i \in 2A+k$ and an even intertwining operator $m^{i,i}_z \in \cI_{i,i}^{2i}$.
\end{cor}
\begin{proof}
By Lemma \ref{lem:forms-with-restricted-values}, any $\pm 1$-valued quadratic form on $A$ is identically 1 on $2A$, and constant on $2A+k$.  Thus, evenness of one intertwining operator is equivalent to evenness of the commutativity datum.
\end{proof}

\subsection{Modules and intertwining operators}

In addition to building vertex algebras from parts, we would like to build new modules and consider intertwining operators between them.

\begin{defn}
Let $V$ be a $\cT$-vertex algebra, let $A$ be an abelian group, let $\{ \cI_{i,j}^{i+j} \}_{i,j \in A}$ be an even $\cT$-commutativity datum with modules $\{M_i \}_{i \in A}$, let $S$ be a set equipped with an $A$-action (written as addition), and let $\{M_s\}_{s \in S}$ be a set of $V$-modules in $\cT$.  We define a $\cT$-module datum to be a set $\{ \cI_{i,s}^{i+s} \}_{i \in A, s \in S}$ of one-dimensional vector spaces, whose elements are $\cT$-intertwining operators $M_i \otimes M_s \to M_{i+s}((z))$, such that:
\begin{enumerate}
\item For each $i,j \in A$, $s \in S$, there exists a one dimensional vector space $\cI_{i,j,s}^{i+j+s}$ whose elements are maps $M_i \otimes M_j \otimes M_s \to M_{i+j+s}[[w,z]][w^{-1}, z^{-1}, (w-z)^{-1}]$, such that composition of $\cT$-intertwining operators, combined with the canonical injections from Lemma \ref{lem:basic-embeddings}, induces isomorphisms from $\cI^{i+j+s}_{i,j+s} \otimes \cI^{j+s}_{j,s}$ to $\cI_{i,j,s}^{i+j+s}$ and from $\cI^{i+j+s}_{i+j,s} \otimes \cI_{i,j}^{i+j}$ to $\cI_{i,j,s}^{i+j+s}$.  
\item For any $i,j,k\in A$, $s \in S$, there exists a one-dimensional space $\cI_{i,j,k,s}^{i+j+k+s}$ whose elements are maps:
\[ M_i \otimes M_j \otimes M_k \otimes M_s \to M_{i+j+k+s}[[z,w,t]][z^{-1}, w^{-1} t^{-1}, (z-w)^{-1}, (z-t)^{-1}, (w-t)^{-1}] \]
such that composition of intertwining operators, combined with one of the canonical injections from Lemma \ref{lem:star-of-embeddings}, induces an isomorphism from $\cI^{i+j+k+s}_{i,j+k+s} \otimes \cI^{j+k+s}_{j,k+s} \otimes \cI^{k +s}_{k,s}$ to $\cI_{i,j,k,s}^{i+j+k+s}$.
\item The space $\cI_{0,s}^s$ is spanned by the module structure map $act^s_z: V \otimes M_s \to M_s((z))$.
\end{enumerate}

Suppose we are given a choice of nonzero intertwining operators $\{ m^{i,j}_z\}$ inducing a $\cT$-vertex algebra structure on $\bigoplus_{i \in A} M_i$.  Given a $\cT$-module datum, and any choice of nonzero intertwining operators $\{ m^{i,s}_z \in \cI_{i,s}^{i+s} \}_{i \in A, s \in S}$, we define the function $\Phi: A^{\oplus 2} \times S \to \bC^\times$ by
\[ m^{i,j+s}_z \circ (1 \otimes m^{j,s}_w) = \Phi(i,j,s) m^{i+j,s}_w \circ (m^{i,j}_{z-w} \otimes 1) \]
\end{defn}

\begin{lem} \label{lem:module-datum-cocycle}
Given a $\cT$-module datum, and any choice of nonzero intertwining operators $\{ m^{i,s}_z \in \cI_{i,s}^{i+s} \}_{i \in A, s \in S}$, the function $\Phi$ satisfies the identity: 
\[ \Phi(i,j+k,s) \Phi(j,k,s) = \Phi(i+j,k,s) \Phi(i,j,k+s) \]
for any $i,j,k \in A, s \in S$.  In other words, $\Phi$ is a 2-cocycle for $A$ with coefficients in the $A$-module $Fun(S, \bC^\times)$.  The group of 1-cochains $\{ \lambda_{i,s} \}_{i \in A, s \in S}$ (which can be viewed as maps $A \times S \to \bC^\times$) acts by rescaling intertwining operators, and it acts transitively on the set of 2-cocycles in the same cohomology class, by coboundary translation.  The 2-cocycle $\Phi$ is homologically trivial if and only if there exists an $S$-graded $\bigoplus_{i \in A} M_i$-module structure on $\bigoplus_{s \in S} M_s$ whose component structure maps span the $\cT$-module datum.  The action of 1-cochains, when restricted to 1-cocycles, induces an action of $H^1(A, Fun(S,\bC^\times))$ on isomorphism classes of $S$-graded modules with fixed graded pieces.  The action of $0$-cochains by rescaling modules induces an action of the 0-cocycle group $Fun(S,\bC^\times)^A$ by automorphisms of the $S$-graded module.
\end{lem}
\begin{proof}
By essentially the same argument as the proof of Lemma \ref{lem:pentagon}, we have
\[ F(i,j,k) \Phi(i,j+k,s) \Phi(j,k,s) = \Phi(i+j,k,s) \Phi(i,j,k+s) \]
However, we assume $F(i,j,k) = 1$ for all $i,j,k \in A$, so we obtain the identity we want.

By examining the definition of $\Phi$, we see that rescaling the intertwining operators by some 1-cochain $\{ \lambda_{i,s} \}_{i \in A, s \in S}$ multiplies $\Phi$ by the coboundary $d\lambda$, where $(d\lambda)(i,j,s) = \frac{\lambda_{i,j+s} \lambda_{j,s}}{\lambda_{i+j,s}}$.

$\Phi$ is homologically trivial if and only if there is a 1-cochain $\lambda$ such that $d\lambda \cdot \Phi = 0$, which is equivalent to the maps $\{ \lambda_{i,s} m^{i,s}_z\}$ inducing a module structure.

If $\Phi$ is homologically trivial, then translation by the coboundary of a 1-cocycle $\{\lambda_{i,s} \}$ preserves the triviality of the $\Phi$, so we have an action of 1-cocycles on isomorphism classes of modules.  1-coboundaries lie in the kernel of this action, since they amount to rescaling pieces of the module, i.e., there is a block diagonal matrix that induces a module isomorphism.  The automorphisms induced by 0-cocycles are just rescalings on $A$-orbits in $S$.
\end{proof}

\begin{defn}
Let $V$ be a $\cT$-vertex algebra, let $A$ be an abelian group, let $M_A = \bigoplus_{i \in A}$ be a $\cT$-vertex algebra given by a $\cT$-commutativity datum on $V$ modules $\{M_i \}_{i \in A}$, let $+: S_1 \times S_2 \to S_3$ be an $A$-equivariant map of $A$-sets (written as addition), and let $M_{S_1}, M_{S_2}, M_{S_3}$ be $M_A$-modules given by $\cT$-module data.  We define a $\cT$-intertwining operator datum to be a set $\{ \cI_{r,s}^{r+s} \}_{r \in S_1, s \in S_2}$ of one-dimensional vector spaces, whose elements are $\cT$-intertwining operators $M_r \otimes M_s \to M_{r+s}((z))$, such that:
\begin{enumerate}
\item For each $i \in A$, $r \in S_1, s \in S_2$, there exists a one dimensional vector space $\cI_{i,r,s}^{i+r+s}$ whose elements are maps $M_i \otimes M_r \otimes M_s \to M_{i+r+s}[[w,z]][w^{-1}, z^{-1}, (w-z)^{-1}]$, such that composition of intertwining operators, combined with the canonical injections from Lemma \ref{lem:basic-embeddings}, induces isomorphisms from $\cI^{i+r+s}_{i,r+s} \otimes \cI^{r+s}_{r,s}$ to $\cI_{i,r,s}^{i+r+s}$ and from $\cI^{i+r+s}_{i+r,s} \otimes \cI_{i,r}^{i+r}$ to $\cI_{i,r,s}^{i+r+s}$.  
\item For any $i,j\in A$, $r \in S_1$, $s \in S_2$, there exists a one-dimensional space $\cI_{i,j,r,s}^{i+j+r+s}$ whose elements are maps:
\[ M_i \otimes M_j \otimes M_r \otimes M_s \to M_{i+j+r+s}[[z,w,t]][z^{-1}, w^{-1} t^{-1}, (z-w)^{-1}, (z-t)^{-1}, (w-t)^{-1}] \]
such that composition of intertwining operators, combined with one of the canonical injections from Lemma \ref{lem:star-of-embeddings}, induces an isomorphism from $\cI^{i+j+r+s}_{i,j+r+s} \otimes \cI^{j+r+s}_{j,r+s} \otimes \cI^{r +s}_{r,s}$ to $\cI_{i,j,r,s}^{i+j+r+s}$.
\end{enumerate}

Given a $\cT$-intertwining operator datum and any choice of nonzero intertwining operators $\{ m^{r,s}_z \in \cI_{r,s}^{r+s} \}_{r \in S_1, s \in S_2}$, we define the function $\Psi: A \times S_1 \times S_2 \to \bC^\times$ by 
\[ m^{i,r+s}_z \circ (1 \otimes m^{r,s}_w) = \Psi(i,r,s) m^{i+r,s}_w \circ (m^{i,r}_{z-w} \otimes 1) \]
\end{defn}

\begin{lem} \label{lem:intertwining-operator-datum-cocycle}
Given a $\cT$-intertwining operator datum, and any choice of nonzero intertwining operators $\{ m^{r,s}_z \in \cI_{r,s}^{r+s} \}_{r \in S_1, s \in S_2}$, the function $\Psi$ satisfies the identity:
\[ \Psi(i,j+r,s) \Psi(j,r,s) = \Psi(i+j,r,s) \]
In particular, $\Psi$ is a 1-cocycle for $A$ with coefficients in the $A$-module $Fun(S_1 \times S_2, \bC^\times)$, where the $A$-action on $S_1 \times S_2$ is by $a + (s_1,s_2) = (a+s_1,s_2)$, i.e., we choose a trivial $A$-action on $S_2$.  Rescaling the intertwining operators $m^{r,s}_z$ induces an action of the group of 0-cochains, by 1-coboundary translations.  0-cocycles are precisely the elements that leave $\Psi$ unchanged, and in particular, the degree zero cohomology group $Fun(S_1 \times S_2, \bC^\times)^A$ acts on the space of intertwining operators assembled from $\{m^{r,s}_z\}$.
\end{lem}
\begin{proof}
By essentially the same argument as the proof of Lemma \ref{lem:pentagon}, we find that
\[ \Phi(i,j,r) \Psi(i,j+r,s) \Psi(j,r,s) = \Psi(i+j,r,s) \Phi(i,j,r+s) \]
Since we assume $\Phi$ is identically 1, we obtain the identity we want.  
By examining the definition of $\Psi$, we see that rescaling the intertwining operators by some 0-cochain $\{ \lambda_{r,s} \}_{r \in S_1, s \in S_2}$ multiplies $\Psi$ by the coboundary $d\lambda$, where $(d\lambda)(i,r,s) = \frac{\lambda_{i+r,s}}{\lambda_{r,s}}$.
\end{proof}

\begin{prop} \label{prop:formation-of-modules}
Let  $\{ \cI_{i,j}^{i+j} \}_{i,j \in A}$ be a $\cT$-commutativity datum, and let $B \subset A$ be a subgroup such that the quadratic form $i \mapsto \Omega(i,i)$ is constant on cosets of $B$.  Then for any choice of $\cT$-vertex algebra $M_B$ attached to the commutativity datum, the coset sums $M_{B+k} = \bigoplus_{i \in B+k} M_i$ admit the structure of $M_B$-modules in $\cT$, such that for each $i,j \in A$, the restriction of the action to $M_i \otimes M_j$ lands $\cI_{i,j}^{i+j}$.  Furthermore, there exists a one-dimensional commutativity datum $\{ \cI_{B+i,B+j}^{B+i+j} \}$ on $\{M_{B+k} \}$ whose restriction to $\{ M_i \otimes M_j \}$ spans the commutativity datum $\{ \cI_{i,j}^{i+j} \}_{i,j \in A}$.
\end{prop}
\begin{proof}
By Lemma \ref{lem:pullback-quadratic-form}, $i \mapsto \Omega(i,i)$ is the pullback of a quadratic form on $A/B$, and there is a representing normalized cocycle $(F,\Omega)$ pulled back from a normalized abelian 3-cocycle $(\bar{F}, \bar{\Omega})$ on $A/B$.
By the transitivity claim in Theorem \ref{thm:obstruction}, there exists a choice of intertwining operators $m^{i,j}_z$ representing $(F,\Omega)$.  We claim that from this choice, we obtain a module structure on $M_{B+k}$ for each coset $B + k$, and intertwining operators between such modules that span a commutativity datum.

For the module structure, the obstruction 2-cocycle $\Phi$ vanishes, as it is the restriction of $F$ to $B \times B \times (B+k)$, hence the pullback of the normalized cocycle $\bar{F}$ on $(0,0,k)$.

To show that a map $M_{B+i} \otimes M_{B+j} \to M_{B+i+j}((z))$ is an intertwining operator, we note that the obstruction 1-cocycle $\Psi$ vanishes, as it is the restriction of $F$ to $B \times (B+i) \times (B+j)$, hence the pullback of the normalized cocycle $\bar{F}$ on $(0,i,j)$.

To see that these intertwining operators span an associativity datum, we note that each space $\cI_{B+i,B+j,B+k}^{B+i+j+k}$ is naturally spanned by composites of the intertwining operators we constructed, and the composites differ by the scalar $\bar{F}(B+i,B+j,B+k)$.  To see that this associativity datum is a commutativity datum for $\{M_{B+i}\}_{B+i \in A/B}$, we note that skew-symmetry follows from the normalization of $\bar{\Omega}$, i.e., $\bar{\Omega}(0,B+i) = 1$.
\end{proof}

\begin{cor} \label{cor:2-torsion-modules-exist}
Any $\cT$-commutativity datum for $\{M_i\}_{i \in A}$, induces for each coset $2A+k$ in $A$, and each associated $\cT$-vertex algebra structure $M_{2A}$, an $M_{2A}$-module structure on the sum $M_{2A+k} = \bigoplus_{i \in 2A+k} M_i$, and a $\cT$-commutativity datum for the resulting spaces of intertwining operators.
\end{cor}
\begin{proof}
By Lemma \ref{lem:forms-with-restricted-values}, any $\pm 1$-valued quadratic form on $A$ is identically 1 on $2A$, and constant on cosets of $2A$.  The $\cT$-commutativity datum then satisfies the conditions of Proposition \ref{prop:formation-of-modules}.
\end{proof}

\section{Quasi-simple current extensions} \label{sect:simple-currents}

Now that we have a good theory for constructing vertex algebras and modules from pieces, we'd like to prove existence theorems with fewer hypotheses about the concrete structure of intertwining operators, since they are in general difficult to compute.  Instead, we shall make increasingly strong assumptions on the modules and their categories, and we will find that we can make much stronger conclusions about existence.  Simple currents are ideal for this, since the intertwining operator structure is somehow minimal, and the commutativity datum exists essentially automatically.

Since we are considering vertex algebras that are not necessarily rational vertex operator algebras, we use slightly more general language, appending the prefix ``quasi'' make the distinction.  In particular, our quasi-simple currents are not quite as well-behaved, as they lack some of the finiteness assumptions that appear in the literature.

In addition to our type $\cT$, we will now introduce a new datum, restricting our view to a full subcategory $\cC$ of simple $V$-modules in $\cT$.  This will help us manage fusion, since the full structure of the $V$-module category may be beyond our reach.  Choosing $\cC$ in concrete situations can be rather delicate: if we make $\cC$ too big, our subcategory of quasi-simple currents may collapse to triviality, and if we make $\cC$ too small, we may not get interesting phenomena in the first place.

\subsection{Restricted quasi-simple currents}

\begin{defn} \label{defn:quasi-simple}
Suppose $V$ is a simple $\cT$-vertex algebra, and let $\cC$ be a full subcategory of the category of $V$-modules in $\cT$, such that all objects in $\cC$ are simple $V$-modules.  We define a quasi-simple $\cC$-current to be an object $M$ in $\cC$ such that for any object $N$ in $\cC$, there exists an object $X$ in $\cC$ and a nonzero $\cT$-intertwining operator $m^{M,N}_z: M \otimes N \to X((z))$, such that substitution of coefficients induces isomorphisms $\Hom_V(X,Y) \cong I_{\cT}\binom{Y}{M,N}$ for all objects $Y$ in $\cC$.  In this case, we write $M \boxtimes N$ (or $M \boxtimes_{\cC} N$ if there is the possibility of confusion) to denote the object $X$, and call the pair $(M \boxtimes N, m^{M,n}_z)$ the relative fusion product of $M$ and $N$ in $\cC$.
\end{defn}

\begin{lem}
The isomorphism classes of objects in $\cC$ form a set.
\end{lem}
\begin{proof}
Any irreducible $V$-module is an irreducible module for the Frenkel-Zhu enveloping associative algebra of $V$ \cite{FZ92}.  Isomorphism classes of irreducible modules of an associative algebra can be identified with a subquotient of the power set of the algebra, since any such module is generated by some vector $m$, and hence isomorphic to the quotient by the maximal left ideal $Ann(m)$.
\end{proof}

\begin{rem}
Let $M$ be a quasi-simple current, and let $X = M \boxtimes N$ as in Definition \ref{defn:quasi-simple}.  Skew-symmetry implies all intertwining operators $N \otimes M \to X((z))$ are scalar multiples of $m^{M,N,*}_z$.  In particular, we have canonical isomorphisms $\Hom_V(X,Y) \cong I_{\cT}\binom{Y}{N,M}$.
\end{rem}

\begin{defn}
A pair $(M,N)$ of quasi-simple $\cC$-currents is called composable if the following conditions hold:
\begin{enumerate}
\item $M \boxtimes N$ is a quasi-simple $\cC$-current.
\item For any object $P$ in $\cC$, $M \boxtimes (N \boxtimes P) \cong (M \boxtimes N) \boxtimes P$.
\item For any object $P$ in $\cC$, the composites $m^{M, N \boxtimes P}_z \circ (1 \otimes m^{N,P}_w)$ and $m^{M \boxtimes N,P}_w \circ (m^{M,N}_{z-w} \otimes 1)$ factor through the inclusions from $(M \boxtimes (N \boxtimes P))[[z,w]][z^{-1},w^{-1},(z-w)^{-1}]$ to $(M \boxtimes (N \boxtimes P))((z))((w))$ and $(M \boxtimes (N \boxtimes P))((w))((z-w))$, and the two maps $M \otimes N \otimes P \to M \boxtimes (N \boxtimes P)[[z,w]][z^{-1},w^{-1},(z-w)^{-1}]$ are scalar multiples of each other.
\end{enumerate}
\end{defn}

\begin{lem} \label{lem:composable-pairs-give-associativity-and-commutativity}
If $(M_1,M_2)$ is a composable pair, then for any irreducible module $M_3$, the following four modules are isomorphic: $M_1 \boxtimes (M_2 \boxtimes M_3)$, $(M_1 \boxtimes M_2) \boxtimes M_3$, $(M_2 \boxtimes M_1) \boxtimes M_3$, $M_2 \boxtimes (M_1 \boxtimes M_3)$.
\end{lem}
\begin{proof}
By skew-symmetry, $M_1 \boxtimes M_2 \cong M_2 \boxtimes M_1$, so the notion of composable pair is independent of order.  In particular, the second and third modules are isomorphic.  By the definition of composable pair, the first two modules are isomorphic, and the last two modules are hence isomorphic.
\end{proof}

\begin{lem} \label{lem:composable-pairs-give-locality}
If $(M_1,M_2)$ is a composable pair, then for any irreducible module $M_3$ and elements $m_i \in M_i$, $m^{M_1,M_2 \boxtimes M_3}_z \circ (1 \otimes m^{M_2,M_3}_w)(m_1 \otimes m_2 \otimes m_3)$ is a constant multiple of $m^{M_2,M_1 \boxtimes M_3}_w \circ (1 \otimes m^{M_1,M_3}_z)(m_2 \otimes m_1 \otimes m_3)$ in $M_1 \boxtimes (M_2 \boxtimes M_3)[[z,w]][z^{-1},w^{-1},(z-w)^{-1}]$.  In particular, there is some positive integer $N$ such that multiplication by $z^N w^N (z-w)^N$ gives proportional elements in $M_1 \boxtimes (M_2 \boxtimes M_3)[[z,w]]$.
\end{lem}
\begin{proof}
This is essentially the standard argument that skew-symmetry and associativity imply locality.  
\end{proof}

\begin{defn}
A quasi-simple $\cC$-current is universally composable if it forms composable pairs with all quasi-simple $\cC$-currents.  We write $Curr_{\cC}(V)$ for the full subcategory of universally composable quasi-simple $\cC$-currents.
\end{defn}


\begin{lem} \label{lem:4-input-intertwining-operators}
Suppose $M_1$ and $M_2$ are universally composable quasi-simple $\cC$-currents, and $M_3$ is a quasi-simple $\cC$-current.  For any object $P$ in $\cC$, and any permutation $\sigma$ of $\{1,2,3\}$, let $X_\sigma = M_{\sigma(1)} \boxtimes (M_{\sigma(2)} \boxtimes (M_{\sigma(3)} \boxtimes P))$.  Then all $X_\sigma$ are isomorphic, and all composites
\[ m^{M_{\sigma(1)}, M_{\sigma(2)} \boxtimes (M_{\sigma(3)} \boxtimes P)}_{z_{\sigma(3)}} \circ (1 \otimes m^{M_{\sigma(2)}, M_{\sigma(3)} \boxtimes P}_{z_{\sigma(2)}}) \circ (1 \otimes 1 \otimes m^{M_{\sigma(3)},P}_{z_{\sigma(1)}}) \]
factor through the embeddings
\[ X_1[[z_1,z_2,z_3]][z_1^{-1},z_2^{-1},z_3^{-1},(z_1-z_2)^{-1},(z_1-z_3)^{-1},(z_2-z_3)^{-1}] \to X_\sigma((z_{\sigma(3)}))((z_{\sigma(2)}))((z_{\sigma(1)})). \]
Furthermore, the induced maps $M_1 \otimes M_2 \otimes M_3 \otimes P \to X_1[[z_1,z_2,z_3]][z_1^{-1},z_2^{-1},z_3^{-1},(z_1-z_2)^{-1},(z_1-z_3)^{-1},(z_2-z_3)^{-1}]$ are scalar multiples of each other.
\end{lem}
\begin{proof}
The isomorphisms between the $X_\sigma$ follow from repeatedly applying Lemma \ref{lem:composable-pairs-give-associativity-and-commutativity}.

For each permutation $\sigma$, we consider an element $p_\sigma = m_{\sigma(1)} \otimes m_{\sigma(2)} \otimes m_{\sigma(3)} \otimes p \in M_{\sigma(1)} \otimes M_{\sigma(2)} \otimes M_{\sigma(3)} \otimes P$, and let $x_\sigma = m^{M_{\sigma(1)}, M_{\sigma(2)} \boxtimes (M_{\sigma(3)} \boxtimes P)}_{z_{\sigma(3)}} \circ (1 \otimes m^{M_{\sigma(2)}, M_{\sigma(3)} \boxtimes P}_{z_{\sigma(2)}}) \circ (1 \otimes 1 \otimes m^{M_{\sigma(3)},P}_{z_{\sigma(1)}})(p_\sigma)$

Locality, as in Lemma \ref{lem:composable-pairs-give-locality}, implies that for a sufficiently large positive integer $N$, the series $(z_1-z_2)^N(z_1-z_3)^N(z_2-z_3)^Nx_\sigma$ are independent of $\sigma$ up to scalar multiplication.  In particular, poles in the variables $z_1, z_2, z_3$ are uniformly bounded in order.  This implies the existence of the factorization.
\end{proof}

\begin{lem}
The fusion product of two universally composable quasi-simple $\cC$-currents is universally composable.
\end{lem}
\begin{proof}
Let $M_1$, $M_2$, and $M_3$ be quasi-simple $\cC$-currents, and suppose $M_1$ and $M_2$ are universally composable.  We wish to show that $M_1 \boxtimes M_2$ and $M_3$ form a composable pair.

For the first condition, we need to show that $(M_1 \boxtimes M_2) \boxtimes M_3$ is a quasi-simple $\cC$-current, and by Lemma \ref{lem:composable-pairs-give-associativity-and-commutativity}, it suffices to show that $M_1 \boxtimes (M_2 \boxtimes M_3)$ is a quasi-simple $\cC$-current.  Since $M_2$ is universally composable, $M_2 \boxtimes M_3$ is a quasi-simple $\cC$-current, and since $M_1$ is universally composable, $M_1 \boxtimes (M_2 \boxtimes M_3)$ is a quasi-simple $\cC$-current.

For the second condition, we alternately use the universally composable properties of $M_1$ and $M_2$ to obtain the following isomorphisms:
\[ \begin{aligned}
(M_1 \boxtimes M_2) \boxtimes (M_3 \boxtimes P) &\cong M_1 \boxtimes (M_2 \boxtimes (M_3 \boxtimes P)) \\
&\cong M_1 \boxtimes ((M_2 \boxtimes M_3) \boxtimes P) \\
&\cong (M_1 \boxtimes (M_2 \boxtimes M_3)) \boxtimes P \\
&\cong ((M_1 \boxtimes M_2) \boxtimes M_3) \boxtimes P
\end{aligned} \]

For the last condition, we shall write $M' = ((M_1 \boxtimes M_2) \boxtimes M_3) \boxtimes P$ for convenience.  We need to show that the composites $m^1 = m^{M_1 \boxtimes M_2, M_3 \boxtimes P}_w \circ (1 \otimes m^{M_3,P}_t)$ and $m^2 = m^{(M_1 \boxtimes M_2)  \boxtimes M_3,P}_t \circ (m^{M_1 \boxtimes M_2,M_3}_{w-t} \otimes 1)$ factor through the inclusions from $M'[[w,t]][w^{-1},t^{-1},(w-t)^{-1}]$ to $M'((w))((t))$ and $M'((t))((w-t))$, and that the two maps $(M_1 \boxtimes M_2) \otimes M_3 \otimes P \to M'[[w,t]][w^{-1},t^{-1},(w-t)^{-1}]$ are scalar multiples of each other.  By the irreducibility of $M_1 \boxtimes M_2$, it suffices to show this for all coefficients of $(z-w)^n$ in $m^{M_1,M_2}_{z-w} \otimes 1 \otimes 1$ applied to $M_1 \otimes M_2 \otimes M_3 \otimes P$.

We use the notation of Lemma \ref{lem:star-of-embeddings}, but with $M'$ in place of $V$.  Pre-composing both $m^1$ and $m^2$ with $m^{M_1,M_2}_{z-w}$, we obtain a map $\hat{m}^1$ to $M'_{(ij)(k\ell)}$ and a map $\hat{m}^2$ to $M'_{((ij)k)\ell}$.  It therefore suffices to show that both factor through maps to $M'_{[(ij)k]\ell}$ that are scalar multiples of each other.

By Lemma \ref{lem:4-input-intertwining-operators}, the composite $\hat{m}^3 = m^{M_1, M_2 \boxtimes (M_3 \boxtimes P)}_z \circ (1 \otimes m^{M_2,M_3 \boxtimes P}_w) \circ (1 \otimes 1 \otimes m^{M_3,P}_t)$ factors through $M'_{[ijk]\ell}$.  By repeatedly applying Lemma \ref{lem:composable-pairs-give-locality}, we see that $\hat{m}^1$ and $\hat{m}^2$ also factor through maps to $M'_{[ijk]\ell}$, and are proportional to the corresponding map that factors $\hat{m}^3$.  Using the embedding $M'_{[ijk]\ell} \to M'_{[(ij)k]\ell}$, we see that $\hat{m}^1$ and $\hat{m}_2$ factor through maps to $M'_{[(ij)k]\ell}$ that are proportional.
\end{proof}

\begin{defn}
A $\cT$-vertex algebra $V$ is a $\cC$-unit if it is an object in $\cC$, and if it satisfies $V \boxtimes M \cong M$ for any object $M$ in $\cC$.  In particular, it is a quasi-simple $\cC$-current.
\end{defn}

\begin{prop}
Suppose $V$ is a $\cC$-unit.  Then $Curr_{\cC}(V)$ is a $\bC$-linear symmetric monoidal category under fusion product $(M, N) \mapsto M \boxtimes N$.
\end{prop}
\begin{proof}
The associator $a_{MNP}: (M \boxtimes N) \boxtimes P \to M \boxtimes (N \boxtimes P)$ is the proportionality constant between the three-input intertwining operators given by composites.

By combining Lemma \ref{lem:4-input-intertwining-operators} and Lemma \ref{lem:composable-pairs-give-locality}, the compositions of intertwining operators given in the pentagon diagram are proportional, so the associator satisfies the pentagon axiom.

The unit structures $\ell_M: V \boxtimes M \to M$ and $r_M: M \boxtimes V \to M$ are canonically given by the action of $V$, and the intertwining operator induced by skew-symmetry.

The identity $r_M \otimes 1 = (1 \otimes \ell_N) \circ a_{MVN}$ is straightforward.

The commutor is the proportionality constant between $m^{M,N}_z$ and $m^{N,M,*}_z$.  This is automatically symmetric, since intertwining operators are integer weight.

By  combining Lemma \ref{lem:4-input-intertwining-operators} and Lemma \ref{lem:composable-pairs-give-locality}, the compositions of intertwining operators given in the hexagon diagrams are proportional, so the commutor and associator satisfy the hexagon axioms.
\end{proof}

From now on, we assume $V$ is a $\cC$-unit.

\begin{defn}
An inverse of a universally composable quasi-simple $\cC$-current $M$ is a universally composable quasi-simple $\cC$-current $N$, such that $M \boxtimes N \cong V$ for $V$ a $\cC$-unit.  If $M$ admits an inverse, we say that $M$ is invertible.  We define the Picard groupoid $\uPic_{\cC}(V)$ to be the category of invertible universally composable quasi-simple $\cC$-currents, with morphisms given by isomorphisms in $\cC$.
\end{defn}


\begin{lem}
$\uPic_{\cC}(V)$ is a symmetric monoidal subcategory of $Curr_{\cC}(V)$.
\end{lem}
\begin{proof}
It suffices to show that the fusion product of two invertible objects is invertible.  If $M$ and $N$ are invertible universally composable $\cC$-currents, then the associativity property of composable pairs implies $M \boxtimes N$ and $M^{-1} \boxtimes N^{-1}$ are inverses.
\end{proof}

\begin{rem}
Note that this Picard groupoid is not necessarily a strictly commutative Picard category in the sense of \cite{SGA4} Exp. XVIII Definition 1.4.2, since the commutor isomorphism on $M \boxtimes M$ is not necessarily identity.
\end{rem}

\begin{lem}
The isomorphism classes in $\uPic_{\cC}(V)$ form a set, and the monoidal structure induces an abelian group structure on this set.
\end{lem}
\begin{proof}
The isomorphism classes in $\uPic_{\cC}(V)$ are a subclass of the set of isomorphism classes in $\cC$, so they form a set.  $\pi_0$ of an essentially small symmetric monoidal category is an abelian monoid, and invertibility of objects implies invertibility of elements of the monoid.
\end{proof}

\begin{defn}
We call the group of isomorphism classes in $\uPic_{\cC}(V)$ the Picard group of $V$ in $\cC$, and write $\Pic_{\cC}(V)$.  Given a homomorphism $A \to \Pic_{\cC}(V)$ of abelian groups, we write $M_i$ for a module in $\cT$ in the isomorphism class defined by $\pi(i)$.  We say that a homomorphism $A \to \Pic_{\cC}(V)$ is $\cT$-compatible if the set of modules $\{ M_i \}_{i \in A}$ is $\cT$-compatible (in the sense of Definition \ref{defn:T-compatible}).
\end{defn}

\begin{prop} \label{prop:simple-currents-commutativity-datum}
Let $\pi: A \to \Pic_{\cC}(V)$ be a homomorphism of abelian groups.  Then the modules $\{ M_i \}_{i \in A}$ admit a canonical $\cT$-commutativity datum.
\end{prop}
\begin{proof}
Our spaces of intertwining operators are spanned by composites.  Lemma \ref{lem:4-input-intertwining-operators} gives us our associativity datum.  Skew-symmetry is more or less automatic.  
\end{proof}

\begin{defn}
Let $\pi: A \to \Pic_{\cC}(V)$ be a homomorphism of abelian groups.  Any $\cT$-vertex algebra structure on $\bigoplus_{i \in A} M_i$ whose multiplication restricts to nonzero $\cT$-intertwining operators $M_i \otimes M_j \to M_{i+j}((z))$ (which is unique up to isomorphism if it exists) is called the quasi-simple $\cC$-current extension of $V$ by $\pi$.
\end{defn}

\begin{defn}
A quasi-simple $\cC$-current $M$ is even if some (equivalently, any) nonzero intertwining operator $M \otimes M \to (M \boxtimes M)((z))$ is even.  A homomorphism $A \to \Pic_{\cC}(V)$ is even if the image is spanned by even objects.
\end{defn}

\begin{rem}
At our current level of generality, it seems that the even objects in the Picard groupoid of $V$ do not necessarily form a monoidal subcategory.  At least, if $\Pic_{\cC}(V) \cong \bZ/2\bZ \times\bZ/2\bZ$, I was unable to rule out the possibility that the quadratic form $\Omega$ is $-1$ on $(1,1)$ and $1$ on all other elements.
\end{rem}

\begin{prop} \label{prop:even-homs-yield-vertex-algebras}
For any even homomorphism $\pi: A \to \Pic_{\cC}(V)$, the canonical commutativity datum given in Proposition \ref{prop:simple-currents-commutativity-datum} is even.  In particular, if $\pi$ is $\cT$-compatible, then there exists a quasi-simple $\cC$-current extension $W_A = \bigoplus_{i \in A} M_i$ of $V$ by $\pi$, and its isomorphism type is uniquely determined.
\end{prop}
\begin{proof}
The condition on $\pi$ forces this commutativity datum to be even.  By Proposition \ref{prop:existence}, $W_A$ admits a $\cT$-vertex algebra structure, whose isomorphism type is uniquely determined.
\end{proof}

\begin{lem}
Let $M$ be an object in $\cC$.  Then the subcategory $\uStab_{\cC}(M)$ of $\uPic_{\cC}(V)$ whose objects $M_i$ satisfy $M_i \boxtimes M \cong M$ forms a Picard subgroupoid.  For any module $M_i$ in $\uPic_{\cC}(V)$, the isomorphism type of $M_i \boxtimes M$ depends only on the coset of the corresponding subgroup $\Stab_{\cC}(M)$ in $\Pic_{\cC}(V)$.
\end{lem}
\begin{proof}
If $M_i$ and $M_j$ stabilize $M$, then $M_i \boxtimes (M_j \boxtimes M) \cong M_i \boxtimes M \cong M$, so we have a full symmetric monoidal subcategory.  If $M_i$ stabilizes $M$, then $M \cong (M_i^{-1} \boxtimes M_i) \boxtimes M \cong M_i^{-1} \boxtimes (M_i \boxtimes M) \cong M_i^{-1} \boxtimes M$, so this category is closed under inverses.  If $M_j$ stabilizes $M$ but $M_i$ does not necessarily stabilize $M$, then $M_i \boxtimes M \cong M_i \boxtimes (M_j \boxtimes M) \cong (M_i \boxtimes M_j) \boxtimes M \cong M_{i+j} \boxtimes M$, so we have independence within cosets.
\end{proof}

\begin{lem} \label{lem:forming-irreducible-modules}
Let $W_A$ denote the quasi-simple $\cC$-current extension of $V$ by an even homomorphism $\pi: A \to \Pic_{\cC}(V)$.  Then for any object $M$ in $\cC$, the direct sum $\bigoplus_{i \in A/(A \cap\Stab_{\cC}(M))} (M_i \boxtimes M)$ is an irreducible $W_A$-module in $\cT$.  In particular, $W_A$ is a simple $\cT$-vertex algebra.
\end{lem}
\begin{proof}
By the coset property of the stabilizer, the direct sum module is multiplicity-free as a $V$-module.  Since each $M_i \boxtimes M$ is irreducible, any $W_A$-submodule $N$, when viewed as a $V$-module, is contained in the direct sum of the irreducible modules $M_i \boxtimes M$, so it is isomorphic to a direct sum of a subset of such modules that is closed under fusion with the constituent modules in $W_A$.  Such a subset is either empty or the full coset space, so $N$ is either the whole direct sum module or zero.
\end{proof}

Recall from Corollary \ref{cor:2-torsion-modules-exist} that for any $\pi: A \to \Pic_{\cC}(V)$, the subgroup $2A$ plays a special role, in that we have existence of a vertex algebra $W_{2A}$, and modules $M_{2A+i}$ for each coset of $2A$ in $A$, such that there is a commutativity datum on these modules induced from the commutativity datum on $V$-modules.  If we choose our category $\cC'$ so that its objects are the $W_{2A}$ modules $M_{2A+i}$, we find that $\cC' = Curr(\cC')$ is symmetric monoidal, and the modules are quasi-simple $\cC'$-currents whose tensor square is $W_{2A}$.  Unfortunately, we cannot conclude that they are even, so we can only conclude that totally even subgroups of $A/2A$ induce $\cT$-vertex algebras as quasi-simple $\cC$-current extensions.

\subsection{Contragradients and regularity}

We continue to use the notation from the previous section, and we maintain the assumption that $V$ is a $\cC$-unit.  From now on, we shall assume that $\cT$ incorporates M\"obius structure (in particular, $V$ is M\"obius, quasi-conformal, conformal, or a vertex operator algebra), and furthermore, that there is a contragradient involution on $\cC$ (see Definition \ref{defn:contragradient} below), such that $V \cong V^\vee$.  Equivalently, we assume there is a nondegenerate invariant bilinear form on $V$.

\begin{defn} \label{defn:contragradient}
Let $V$ be a M\"obius (resp. quasiconformal, conformal) vertex algebra or vertex operator algebra.  Given an integral-weight $V$-module $M$ in the M\"obius (resp. quasiconformal, conformal, VOA) sense, a weak contragradient is a $V$-module $M^\vee$ equipped with a bilinear map $\langle-,-\rangle: M^\vee \times M \to \bC$ satisfying the following properties:
\begin{enumerate}
\item The pairing is non-degenerate, i.e., the left and right kernels are zero.
\item Virasoro operators are adjoint, i.e., $\langle L(i) u, m \rangle = \langle u, L(-i) m \rangle$ for all $i \in \bZ$ for which $L(i)$ and $L(-i)$ are defined.
\item The $V$-action satisfies the adjoint identity from \cite{FHL93} equation 5.2.4:
\[ \langle u, act_z(v \otimes m) \rangle = \langle act_{z^{-1}}(e^{zL(1)}(-z^{-2})^{L(0)} v \otimes u), m \rangle \]
for all $v \in V, m \in M, u \in M^\vee$.
\end{enumerate}
Given a full subcategory $\cC$ of $V$-modules in the M\"obius (resp. quasiconformal, conformal, VOA) sense, a contragradient involution is a contravariant functor on $\cC$ that takes each object to a weak contragradient $M \mapsto M^\vee$ on $\cC$, and whose composite with itself is isomorphic to identity.
\end{defn}

\begin{rem}
For vertex operator algebras, integral-weight modules have a natural contragradient functor given by taking the sum of dual spaces in each $L(0)$-eigenspace - this is involutive because the eigenspaces are finite dimensional.  In more general settings, such a functor may require unnatural choices.  However, \cite{HLZ07} has a ``strongly gradable'' condition that seems to cover the known interesting examples, and implies the existence of a contragradient involution.
\end{rem}

\begin{defn}
Given a $\cT$-intertwining operator $I_z: M_1 \otimes M_2 \to M_3((z))$, its adjoint $I'_z: M_1 \otimes M_3^\vee \to M_2^\vee((z))$ is defined by:
\[ \langle \phi, I_z(v \otimes m) \rangle = \langle I'_{z^{-1}}(e^{zL(1)}(-z^{-2})^{L(0)} v \otimes \phi), m \rangle \]
for all $v \in M_1$, $m \in M_2$, and $\phi \in M_3^\vee$.
\end{defn}

\begin{rem}
By Proposition 5.5.2 in \cite{FHL93}, passage to the adjoint intertwining operator yields an isomorphism $\cI_{\cT}\binom{M_3}{M_1,M_2} \cong \cI_{\cT}\binom{M_2^\vee}{M_1,M_3^\vee}$.
\end{rem}

\begin{lem} \label{lem:contragradients-are-inverses}
For any quasi-simple $\cC$-current $M$, we have $M \boxtimes M^\vee \cong V$.  In particular, the contragradient involution restricts to an involution on $Curr_{\cC}(V)$ and $\uPic_{\cC}(V)$.
\end{lem}
\begin{proof}
Since $M$ is a simple current, $M \boxtimes M^\vee$ is an object of $\cC$.  By using the adjoint intertwining operator of $act^*_z$ and self-duality of $V$, we see that there exists a one-dimensional space of intertwining operators $M^\vee \otimes M \to V((z))$, spanned by the operator $I_z$ defined by:
\[ \langle v, I_z(m' \otimes m) \rangle_V = \langle act_{-z^{-1}}(v \otimes e^{zL(1)}(-z^2)^{-L(0)}m'), e^{z^{-1}L(1)}m \rangle_M \]
for all $v \in V$, $m \in M$, $m' \in M^\vee$.  Thus, $M \boxtimes M^\vee \cong V$, and $M^{-1} \cong M^\vee$.
\end{proof}

\begin{rem}
If $M \cong M^\vee$, then we get a contragradient pairing $\langle-,-\rangle_M$ on $M$, so the space $I_{\cT}\binom{V}{M,M}$ of $\cT$-intertwining operators $I: M \otimes M \to V((z))$ is one dimensional and spanned by the operator $I_z$ given above.  This operator is put to good use in an alternative construction of the monster vertex algebra $V^\natural$, in section 3.4 of \cite{H96}.
\end{rem}

The following is a minor generalization of Proposition 2.6 in \cite{LY08}.

\begin{prop} \label{prop:evenness-equivalent-to-symmetric-form}
Suppose $V$ is self-contragradient. Let $M$ be a self-dual object in $\cC$.  Then $M$ is even if and only if the contragradient pairing $\langle -,- \rangle_M$ is symmetric, and $M$ is odd if and only if the contragradient pairing is antisymmetric.
\end{prop}
\begin{proof}
Let $m$ and $m'$ be homogeneous of weight $k$.  Then we have
\[ \begin{aligned}
\langle \unit, I_z(m' \otimes m) \rangle_V &= \langle act_{-z^{-1}}(\unit \otimes e^{zL(1)}(-z^2)^{-L(0)}m'), e^{z^{-1}L(1)}m \rangle_M \\
&= (-1)^k z^{-2k}\langle e^{zL(1)} m', e^{z^{-1}L(1)} m' \rangle_M \\
&= (-1)^k z^{-2k} \sum_{j=0}^\infty \frac{1}{(j!)^2} \langle L(1)^j m', L(1)^j m \rangle_M \\
\end{aligned}\]
Let $\Omega_M$ be the constant that is $1$ if $M$ is even and $-1$ if $M$ is odd (this is commonly known as the Frobenius-Schur indicator of $M$).  The identity $\Omega_M I_z(m' \otimes m) = I^*_z(m \otimes m') = e^{zL(-1)}I_{-z}(m \otimes m')$ implies
\[ \begin{aligned}
\langle \unit, \Omega_M I_z(m' \otimes m) \rangle_V &= \langle \unit, e^{zL(-1)} I_{-z}(m \otimes m') \rangle_V \\
&= \langle e^{zL(1)}\unit,  I_{-z}(m \otimes m') \rangle_V \\
&= \langle \unit,  I_{-z}(m \otimes m') \rangle_V \\
\end{aligned} \]
We conclude that
\[ \Omega_M \sum_{j=0}^\infty \frac{1}{(j!)^2} \langle L(1)^j m', L(1)^j m \rangle_M = \sum_{j=0}^\infty \frac{1}{(j!)^2} \langle L(1)^j m, L(1)^j m' \rangle_M \]
Note that the restriction of $\langle -,- \rangle_M$ to the primary subspace, where $L(1)$ acts as zero, is clearly symmetric when $M$ is even, and antisymmetric when $M$ is odd.  This is a rather superfluous base of our induction, but it helps to illustrate the point.  For each positive integer $k$, let $M_{[k]}$ denote the subspace of $M$ annihilated by $L(1)^k$, and let $P_k$ be the claim that the restriction of $\langle -,- \rangle_M$ to $M_{[k]}$ is symmetric when $M$ is even, and antisymmetric when $M$ is odd.  If $P_n$ holds for some $n$, then for any $m,m'$ in $M_{[n+1]}$,
\[ \begin{aligned}
\Omega_M \langle m', m \rangle_M + \Omega_M \sum_{j=1}^n \frac{1}{(j!)^2} \langle L(1)^j m', L(1)^j m \rangle_M
&= \Omega_M \sum_{j=0}^n\frac{1}{(j!)^2} \langle L(1)^j m', L(1)^j m \rangle_M \\
&= \sum_{j=0}^n \frac{1}{(j!)^2} \langle L(1)^j m, L(1)^j m' \rangle_M \\
&= \langle m, m' \rangle_M + \sum_{j=1}^n \frac{1}{(j!)^2} \langle L(1)^j m, L(1)^j m' \rangle_M
\end{aligned} \]
By our inductive hypothesis, we can cancel the sums with positive $j$, so $\Omega_M \langle m', m \rangle_M = \langle m, m' \rangle_M$, and $P_{n+1}$ holds.  By the local nilpotence axiom on modules, the increasing filtration of $M$ by $\{M_{[k]} \}_{k \geq 0}$ is exhaustive, so $M$ even implies $\langle -,- \rangle_M$ is symmetric, and $M$ odd implies $\langle -,- \rangle_M$ is antisymmetric.  $M$ is either even or odd, and inner product is nondegenerate, hence it cannot be both symmetric and antisymmetric.  This yields the converse implications.
 \end{proof}

\begin{cor} \label{cor:reduction-to-contragradient-form}
Suppose there exists a contragradient involution on a subcategory $\cC$ of irreducible $V$-modules in $\cT$, and suppose $V$ is self-contragradient.  Let $\pi: A \to \Pic_{\cC}(V)$ be a $\cT$-compatible homomorphism from an $\bF_2$-vector space.  Then the canonical commutativity datum given in Proposition \ref{prop:simple-currents-commutativity-datum} is even if and only if the contragradient form on each module is symmetric.  In particular, there exists a quasi-simple $\cC$-current extension of $V$ by $\pi$ if and only if all of the contragradient forms are symmetric.
\end{cor}
\begin{proof}
The objects in this category are their own contragradients, since the contragradient on $\cC$ permutes the component $V$-modules in each $M_{2A+i}$.  In particular, $V$ is self-contragradient, so by Proposition \ref{prop:evenness-equivalent-to-symmetric-form}, an object in $\cC'$ is even if and only if its contragradient form is symmetric.  Thus, our new commutativity datum is even, and Proposition \ref{prop:existence} yields existence of the quasi-simple $\cC$-current extension of $V$ by $\pi$, if and only if each contragradient form is symmetric..
\end{proof}

At this point, we refine our scope again.  If our M\"obius vertex algebra $V$ satisfies sufficiently strong finiteness hypotheses, then the category of $V$-modules is a braided monoidal category.  Among other benefits, this means we no longer have to make assumptions about the composability of intertwining operators.  Because of the braided monoidal structure, we have a commutative fusion ring $K_0(V)$, whose product encodes the tensor structure of the semisimplification.  Simple modules form a distinguished basis of $K_0(V)$, and simple currents (in the sense of \cite{DLM96}) are the modules represented by units in the fusion ring.  For this purpose, we adopt the following:

\textbf{Convention:} We choose $V$ to be a simple $C_2$-cofinite M\"obius vertex algebra with finite dimensional weight spaces and invariant bilinear form, and $\cC$ to be any full subcategory of the category of $V$-modules, whose objects have integral weight, such that the isomorphism classes span a multiplicative subgroup of the fusion ring.  As a special case, we may take $V$ to be a $C_2$-cofinite vertex operator algebra.

\begin{lem}
If $V$ is a simple $C_2$-cofinite M\"obius vertex algebra with finite dimensional weight spaces, then the category of $V$-modules is braided monoidal, with finitely many simple objects.
\end{lem}
\begin{proof}
By Theorem 12.13 of \cite{HLZ07}, to get a braided monoidal structure, it suffices to show that their Assumption 12.1 holds, but by their Theorem 11.5, it suffices to show that all $V$-modules are $C_1$-cofinite, and the quasi-finite dimensionality condition holds.  Theorem 13 of \cite{GN03} implies both conditions, but under the hypothesis that $V$ is a vertex operator algebra.  However, the proof only uses the assumption that $V$ has a weighted vertex algebra structure with finite dimensional weight spaces and weights that are bounded below.
\end{proof}

\begin{lem}
Under our convention on $\cC$, all objects are invertible universally composable quasi-simple $\cC$-currents.  In particular, $\uPic_{\cC}(V)$ is the core of $\cC$, i.e., the subcategory with the same objects, but only the invertible morphisms.
\end{lem}
\begin{proof}
Because the corresponding elements in the fusion ring are invertible, fusion with any object necessarily takes irreducibles to irreducibles.  Thus, any object of $\cC$ is a quasi-simple $\cC$-current.  Universal composability holds by the compatibility of composites and iterates of intertwining operators (Assumption 12.1 in \cite{HLZ07}).
\end{proof}

\begin{defn}
A unit in the fusion ring of $V$ is called even if it corresponds to an even quasi-simple $\cC$-current.
\end{defn}

\begin{thm} \label{thm:even-units-yield-extension}
Let $V$ be a simple $C_2$-cofinite M\"obius vertex algebra with finite dimensional weight spaces and invariant bilinear form.  For any collection of $V$-modules parametrized by a group of even units in the fusion ring of $V$, the direct sum has the structure of a quasi-simple $\cC$-current extension, and this structure is unique up to isomorphism.  The extension is a $C_2$-cofinite M\"obius vertex algebra with finite dimensional weight spaces.  If $V$ is also a rational vertex operator algebra, then the extension is also rational, i.e., it is a simple regular vertex operator algebra.
\end{thm}
\begin{proof}
The first claim follows from Proposition \ref{prop:even-homs-yield-vertex-algebras}.  $C_2$-cofiniteness follows from Lemma 2.6 in \cite{Y04} - the arguments in the cited papers use only the weighted structure on $V$.  Rationality follows from either Theorem 4.5 of \cite{L01} or Theorem 2.14 of \cite{Y04}.
\end{proof}

As we mentioned in the introduction, we think evenness always holds under our current convention, although the weight of evidence is somewhat stronger in the case of regular VOAs than for $C_2$-cofinite M\"obius vertex algebras.

\begin{conj}[Evenness conjecture]
For any simple self-dual module $M$ over a simple $C_2$-cofinite M\"obius vertex algebra with invariant bilinear form, the contragradient form on $M$ is symmetric.
\end{conj}

\begin{thm} \label{thm:conditional-result}
Assuming the evenness conjecture, any collection of integral-weight modules parametrized by a group of units in the fusion ring of $V$ admits the structure of a quasi-simple $\cC$-current extension on the direct sum.  The extension is a simple $C_2$-cofinite M\"obius vertex algebra, and it is unique up to isomorphism.  If $V$ is also a rational vertex operator algebra, then the extension is also rational, i.e., it is a simple regular vertex operator algebra.
\end{thm}
\begin{proof}
We let $\cC$ denote the subcategory of $V$-modules spanned by our group $A$ of units in $K_0(V)$.  Then Proposition \ref{prop:even-homs-yield-vertex-algebras} implies we have a quasi-simple $\cC$-current extension $W_A$ of $V$, together with a commutativity datum on simple modules $M_{A+i}$.  By the evenness conjecture, the contragradient form on each $M_{A+i}$ is symmetric, so Corollary \ref{cor:reduction-to-contragradient-form} implies the commutativity datum is even, and we obtain a simple current extension of $V$.  The remaining claims follow from Theorem \ref{thm:even-units-yield-extension}.
\end{proof}

\end{document}